\documentclass[11pt,reqno]{amsart}

\usepackage{xcolor}
\usepackage[utf8]{inputenc}
\usepackage[T1]{fontenc}
\usepackage[english]{babel}

\usepackage{microtype}
\usepackage{amsmath,amssymb,amsthm}
\usepackage{amsthm, thmtools, thm-restate}
\usepackage{mathtools} 
\usepackage{bm}        
\usepackage{enumitem}

\usepackage{graphicx}  
\usepackage{caption}
\usepackage{subcaption}
\usepackage{tikz}      
\usepackage{mathrsfs}  
\usepackage{booktabs}  
\usepackage{geometry}
\usepackage{cancel}

\theoremstyle{plain}
\newtheorem{theorem}{Theorem}[section]      
\newtheorem{lemma}[theorem]{Lemma}          
\newtheorem{proposition}[theorem]{Proposition}

\newtheorem{claim}[theorem]{Claim}          

\newtheorem*{theorem*}{Theorem}             

\theoremstyle{definition}

\usepackage[hypertexnames=false]{hyperref} 


\numberwithin{equation}{section}

\newcommand{\R}{\mathbb{R}}

\usepackage{comment}
\usepackage{graphicx}  
\usepackage{caption}
\usepackage{subcaption}
\usepackage{tikz}      
\usepackage{mathrsfs}  
\usepackage{booktabs}  
\usepackage{geometry}
\usepackage{cancel}

\title[Nodal solutions for the Minkowski curvature equation]{Collapse and saturation for radial nodal solutions of a Lorentz--Minkowski mean curvature equation}

\author[A. Boscaggin]{Alberto Boscaggin}

\author[F. Colasuonno]{Francesca Colasuonno}

\author[R. Ziegele]{Ricardo Ziegele}
\address{Alberto Boscaggin, Francesca Colasuonno, Ricardo Ziegele \newline \indent 
Dipartimento di Matematica
\newline\indent
Universit\`a di Torino
\newline\indent
via Carlo Alberto 10, 10123 Torino, Italy}
\email{alberto.boscaggin@unito.it}
\email{francesca.colasuonno@unito.it}
\email{ricardoalfonso.ziegelealiaga@unito.it}

\begin{document}
\maketitle
\begin{abstract}
We consider the Dirichlet problem for the Lorentz--Minkowski mean curvature operator in a radial setting,
$$
-\operatorname{div}\left(\frac{\nabla u}{\sqrt{1-|\nabla u|^2}}\right) 
= \lambda u + \mu g(u) \quad \text{in } \mathcal B_R,
\qquad u = 0 \quad \text{on } \partial\mathcal B_R,
$$
on a ball of radius $R$ centered at the origin $\mathcal B_R \subset \mathbb{R}^N$, where $\lambda, \mu \geq 0$ are parameters, and the nonlinearity $g$ satisfies suitable growth assumptions at zero. We first prove an a priori dichotomy for the asymptotic behaviour, as $\mu \to +\infty$, of solutions with a prescribed number of zeros: either they converge uniformly to $0$, or they converge to a piecewise affine limit function $u_\infty$, with $|u_\infty'|=1$ a.e., having the same nodal properties.  We then show that, when the number of zeros is appropriately prescribed, both alternatives are actually realized.
More precisely, building on the multiplicity result in [A. Boscaggin, F. Colasuonno, and R. Ziegele, \emph{Positive and nodal solutions for the Minkowski mean curvature
equation: multiplicity and asymptotics}, arXiv:2607.15956 (2026)], we refine the shooting construction therein to obtain a family of small solutions converging to zero and, through a backward shooting argument, a family of large solutions converging to a saturated piecewise affine profile.

\end{abstract}

\tableofcontents

\section{Introduction}

In this paper, we are concerned with radial solutions of the boundary value problem
\begin{equation}
\label{eq:main}
\begin{cases}
\displaystyle
-\operatorname{div}\left(\frac{\nabla u}{\sqrt{1-|\nabla u|^2}}\right)
=\lambda u+\mu g(u) & \text{in } \mathcal{B}_R,\\[2mm]
u=0 & \text{on } \partial \mathcal{B}_R,
\end{cases}
\end{equation}
where $\mathcal{B}_R:= \{ x \in  \R^{N} \, \colon \, |x| < R \}$, $\lambda, \mu$ are non-negative parameters and $g$ is a nonlinear term satisfying:
\begin{enumerate}
\item[$(g_1)$] $g: \mathbb{R} \to \mathbb{R}$ is locally Lipschitz continuous;

\item[$(g_2)$] there exist $\delta > 0$, $\gamma > 2$, and $0 < C_1 \leq C_2$ such that $C_1 |s|^\gamma \leq g(s) s \leq C_2 |s|^\gamma$ for every $ s \in (-\delta, \delta)$;

\item[$(g_3)$] $g(s)s > 0$ for every\footnote{Let us note that, due to the bound $\Vert \nabla u \Vert_{L^\infty} \leq 1$, radial solutions of \eqref{eq:main} satisfy $\Vert u \Vert_{L^\infty} \leq R$, and thus assumption $(g_3)$ could be replaced by the weaker requirement that $g(s)s > 0$ for $0 \neq \vert s \vert \leq R$. We adopt this global formulation for convenience in the proofs.} $s \neq 0$.

\end{enumerate}
For definiteness, throughout the paper we restrict our attention to radial
solutions satisfying $u(0)>0$. The case $u(0)<0$ can be treated in an
entirely analogous way, with the signs reversed where appropriate.

As is well known, the differential operator appearing in \eqref{eq:main} is the mean curvature operator for spacelike graphs in Lorentz--Minkowski space, the condition $|\nabla u|<1$ corresponding precisely to the spacelike constraint; 
for the classical geometric treatment of the prescribed mean curvature problem, we refer to the seminal work by Bartnik and Simon \cite{BaSi82} and to the more recent contributions \cite{BoIa19,ByIkMaMa24,DeFiMi26}. On the other hand, from a nonlinear analysis perspective, boundary value problems driven by the Lorentz--Minkowski mean curvature operator have been widely investigated in the last decades, using topological, variational and dynamical systems techniques. Without attempting to give a comprehensive account of the literature, we refer, for instance, to \cite{Az16,BeJeTo13,BeJeMa14,BoCoNo20b,BoGa19,CoObOmRi13b,Da16,Mawhin} and the references therein.

Under assumptions slightly less restrictive than $(g_1)$, $(g_2)$ and $(g_3)$ listed above, problem \eqref{eq:main} was recently considered in our previous paper \cite{BoCoZi26}. Using a shooting technique, we proved there that, for every prescribed number of nodal regions and for $\mu$ sufficiently large, with a threshold depending on the prescribed nodal pattern, radial nodal solutions exist in a characteristic small/large configuration. 
More precisely, if $\ell\geq1$ is any integer and $\mu$ is large enough, a large solution with exactly $\ell-1$ zeros in $(0,R)$ exists for every $\lambda\geq0$, 
whereas a corresponding small solution is obtained under the additional spectral condition
\[
\lambda<\lambda_\ell^{\mathrm{rad}},
\]
where $\lambda_\ell^{\mathrm{rad}}$ denotes the $\ell$-th radial eigenvalue of the Dirichlet Laplacian. 
Thus, the solutions appear essentially in pairs, except that the presence of the linear term $\lambda u$ may remove some of the small solutions. A related paired multiplicity pattern for the purely nonlinear problem, corresponding to $\lambda=0$ in the present setting, had previously been obtained in \cite{BoGa19}.

The terminology \emph{small} and \emph{large} is also suggested by the asymptotic picture obtained in \cite{BoCoZi26} for the corresponding PDE on a general bounded domain $\Omega$. 
In that setting, two distinguished solutions were constructed variationally: a global minimizer of the associated action functional, and a min--max solution. 
Their behaviour as $\mu\to+\infty$ is radically different. The min--max solution converges uniformly to zero, whereas the global minimizer, when assumed to be positive (a fact which is always true when $g$ is odd), converges to the distance function
$u_\infty(x) := d(x,\partial\Omega)$,
thus saturating the spacelike constraint $|\nabla u|\leq1$ in the limit. Indeed, note that the limiting profile $u_\infty$ satisfies the eikonal equation
\[
|\nabla u_\infty|=1 \quad \text{in }\Omega,
\qquad
u_\infty=0 \quad \text{on }\partial\Omega,
\]
in the viscosity sense.

This naturally raises the question addressed in the present paper: \begin{quote} \emph{What is the asymptotic behaviour, as $\mu\to+\infty$, of the radial nodal solutions constructed in \cite{BoCoZi26}?} \end{quote} In particular, one may ask whether the same alternative between collapse to zero and saturation of the gradient constraint persists in the radial oscillatory setting and, in the latter case, what becomes of the nodal structure of the solutions.

Our first result gives a complete a priori answer to this question. Remarkably, no intermediate asymptotic behaviour is possible.

\begin{theorem}
\label{thm:main-dichotomy}
Assume that $(g_1)$, $(g_2)$ and $(g_3)$ hold. Let $\ell\geq1$ and let $\{u_{\mu,\ell}\}_\mu$ be a family of radial solutions of \eqref{eq:main} satisfying $u_{\mu,\ell}(0)>0$ and having exactly $\ell-1$ zeros in $(0,R)$. Then, up to a subsequence, as $\mu\to+\infty$, one of the following alternatives holds:
\begin{enumerate}
\item
$u_{\mu,\ell}\to0$ in $C([0,R])$, and there exists $\bar r \in [0,R)$ such that
\[
u_{\mu,\ell}\to0
\quad\text{in }C^1([\bar r,R]);
\]

\item
$u_{\mu,\ell}\to u_{\infty,\ell}$ in $C([0,R])$, where
$u_{\infty,\ell}$ is a piecewise affine function satisfying
\[
|u'_{\infty,\ell}|=1
\quad\text{a.e. in }[0,R],
\]
which has exactly $\ell-1$ zeros on $(0,R)$, is strictly decreasing in the interval from $r=0$ to its first zero (with $R$ regarded as the first zero when
$\ell=1$), and has exactly one local extremum between any two consecutive zeros.
\end{enumerate}
\end{theorem}

In particular, when $\ell=1$, the nontrivial limiting profile is uniquely determined and is given by
\[
u_{\infty,1}(r)=R-r,
\]
which is nothing but the distance function $x \mapsto d(x,\partial\Omega)$ (with $r = \vert x \vert$) when $\Omega$ is the ball
$\mathcal{B}_R$. Thus, when restricted to positive solutions ($\ell = 1)$, the second alternative recovers the distance-function limit obtained in \cite{BoCoZi26} for the corresponding problem on a general domain.

Theorem~\ref{thm:main-dichotomy} is a purely a priori classification result: the only assumption is that, for each $\mu$, one is given a solution with a fixed nodal pattern; no information on how such solutions are obtained is required. Besides the dichotomy itself, a relevant point is that in the nontrivial case the whole nodal structure is preserved in the singular limit.

Let us briefly describe the main mechanism behind the proof. For a radial solution $u=u(r)$, setting
\[
\varphi(s)=\frac{s}{\sqrt{1-s^2}},
\qquad s\in(-1,1),
\]
and introducing
\[
v(r)=r^{N-1}\varphi(u'(r)),
\]
the radial equation can be written as the first-order system
\begin{equation}
\label{eq:first-order}
\begin{cases}
\displaystyle
u'=\varphi^{-1}\left(\frac{v}{r^{N-1}}\right),\\[2mm]
\displaystyle
v'=-r^{N-1}\bigl(\lambda u+\mu g(u)\bigr).
\end{cases}
\end{equation}
For a family $\{u_{\mu,\ell}\}_{\mu}$ as in Theorem~\ref{thm:main-dichotomy}, the natural rescaling is
\[
w_{\mu,\ell}=\frac{v_{\mu,\ell}}{\mu}.
\]
The universal spacelike bound $|u'_{\mu,\ell}|<1$ provides compactness for $\{u_{\mu,\ell}\}$, while
\[
w'_{\mu,\ell}
=
-r^{N-1}
\left(
\frac{\lambda}{\mu}u_{\mu,\ell}
+
g(u_{\mu,\ell})
\right)
\]
gives compactness for $\{w_{\mu,\ell}\}$. Thus, up to subsequences,
\[
u_{\mu,\ell}\to u_{\infty,\ell},
\qquad
w_{\mu,\ell}\to w_{\infty,\ell}
\]
uniformly on $[0,R]$, and
\[
w'_{\infty,\ell}
=
-r^{N-1}g(u_{\infty,\ell}).
\]
On the other hand,
\[
\varphi(u'_{\mu,\ell})
=
\mu\frac{w_{\mu,\ell}}{r^{N-1}}.
\]
Since
\[
\varphi^{-1}(s)\rightarrow\operatorname{sign}(s)
\qquad\text{as }|s|\to+\infty,
\]
one obtains, at every point where $w_{\infty,\ell}\neq0$,
\[
u'_{\infty,\ell}
=
\operatorname{sign}(w_{\infty,\ell}).
\]
The singular limit is therefore governed by the system
\begin{equation}
\label{eq:limit-system}
\begin{cases}
u'_{\infty,\ell}
=
\operatorname{sign}(w_{\infty,\ell})
&
\text{a.e. on }\{w_{\infty,\ell}\neq0\},
\\[1mm]
w'_{\infty,\ell}
=
-r^{N-1}g(u_{\infty,\ell}),
&
r\in(0,R),
\end{cases}
\tag{1.3}
\end{equation}
together with
\[
w_{\infty,\ell}(0)=0,
\qquad
u_{\infty,\ell}(R)=0.
\]

From the viewpoint of dynamical systems, \eqref{eq:limit-system} is particularly suggestive: away from the set $\{w_{\infty,\ell}=0\}$, it is the dynamics associated with a planar vector field whose first component is $\operatorname{sign}(w_{\infty,\ell})$, and hence is discontinuous across the axis $w_{\infty,\ell}=0$. Although differential systems with discontinuous vector fields admit a general theory of their own, here the specific phase-plane structure of \eqref{eq:limit-system} is enough to obtain the information needed for the asymptotic analysis. In particular, the first equation explains directly the saturation phenomenon: wherever $w_{\infty,\ell}\neq0$, the limiting profile travels with the maximal admissible slope, namely $+1$ or $-1$.

The main issue is therefore to understand what may happen when $w_{\infty,\ell}=0$. A crucial step is to prove that, unless the limiting pair is identically zero,
\[
\bigl(u_{\infty,\ell}(r),w_{\infty,\ell}(r)\bigr)
\neq(0,0)
\qquad
\text{for every }r\in[0,R].
\]
The proof uses in an essential way the fact that the approximating solutions have a uniformly prescribed number of zeros: if a nontrivial limiting trajectory reached the origin, the structure of \eqref{eq:limit-system} would force arbitrarily many sign changes near that point, contradicting the fixed nodal count. Once the origin has been excluded, the zeros of $w_{\infty,\ell}$ can be controlled and the phase-plane rotation of the approximating trajectories passes to the limit. This shows both that
\[
|u'_{\infty,\ell}|=1
\quad\text{a.e. in }[0,R]
\]
and that $u_{\infty,\ell}$ has exactly the same number of zeros as $u_{\mu,\ell}$. The nontrivial limit is consequently a piecewise affine oscillatory profile.

The vanishing alternative requires an additional, more technical ingredient. In this case, we prove that $u'_{\mu,\ell}\to0$ uniformly on a terminal subinterval of $(0,R]$. The argument combines the geometry of the nodal intervals with a monotonicity formula for a suitable energy along radial solutions. This strengthened convergence will also play a crucial role in the construction of the large solutions below.

Theorem~\ref{thm:main-dichotomy} only determines the possible limiting behaviours. Our second main result shows that both alternatives actually occur and can be realized by suitably selecting, among the radial solutions furnished by the shooting argument of \cite{BoCoZi26}, a small and a large family.

\begin{theorem}
\label{thm:realization}
Assume that $(g_1)$, $(g_2)$ and $(g_3)$ hold, and let $\ell\geq1$. Then:
\begin{itemize}
\item
there exist $\bar\mu>0$ and a family
$\{u^{(l)}_{\mu,\ell}\}_{\mu>\bar\mu}$
of solutions of \eqref{eq:main}, with
$u^{(l)}_{\mu,\ell}(0)>0$
and exactly $\ell-1$ zeros in $(0,R)$, such that, up to a subsequence,
\[
u^{(l)}_{\mu,\ell}
\rightarrow
u_{\infty,\ell}
\qquad
\text{uniformly on }[0,R],
\]
where $u_{\infty,\ell}$ is as in alternative~(2) of Theorem~\ref{thm:main-dichotomy} (in particular, $u_{\infty,1}(r) = R-r$);

\item
if, moreover,
$\lambda<\lambda_\ell^{\mathrm{rad}}$,
there exist $\bar\mu>0$ and a family
$\{u^{(s)}_{\mu,\ell}\}_{\mu>\bar\mu}$
of solutions of \eqref{eq:main}, with
$u^{(s)}_{\mu,\ell}(0)>0$
and exactly $\ell-1$ zeros in $(0,R)$, such that
\[
u^{(s)}_{\mu,\ell}
\longrightarrow0
\qquad
\text{uniformly on }[0,R].
\]
\end{itemize}
\end{theorem}

It is tempting to interpret this small/large dichotomy also from a variational perspective. Indeed, for $\ell=1$ the two asymptotic regimes are precisely those exhibited in \cite{BoCoZi26} by the min--max solution and the global minimizer, respectively. Whether a similar variational interpretation can be given to the nodal families for $\ell\geq2$ is an interesting open question. In particular, one may ask whether the large nodal solutions admit some minimizing characterization, while the small ones can be detected through suitable higher min--max constructions.

Although Theorem~\ref{thm:realization} fits naturally with the terminology and the multiplicity picture of \cite{BoCoZi26}, it does not follow directly from the shooting construction developed there. That argument detects the existence of small and large shooting parameters by comparing their angular displacement, but it does not localize the corresponding solutions sharply enough to determine their asymptotic behaviour as $\mu\to+\infty$. A substantial part of the present work is therefore devoted to revisiting the shooting argument with this additional information in mind.

For the small solutions, we refine the shooting construction so as to select a parameter whose initial value tends to zero as $\mu\to+\infty$. Theorem~\ref{thm:main-dichotomy} then forces the corresponding family to fall into the vanishing alternative. The construction relies on a quantitative estimate ensuring a large angular displacement for suitable trajectories at an increasingly small scale. After a rescaling, this reduces to a rotation estimate for a class of planar Hamiltonian systems.

The selection of the large solutions is more delicate. Instead of shooting forward from $r=0$, we perform a \emph{backward shooting} from the boundary $r=R$, prescribing
\[
u(R)=0,
\qquad
v(R)=\eta.
\]
The parameter $\eta$ is then chosen so as to recover the regularity condition $v(0)=0$ and the required number of nodal regions. The construction allows us, at the same time, to keep the terminal momentum $v(R)$ away from zero as $\mu\to+\infty$. If the corresponding solutions converged uniformly to zero, the $C^1$-information in the first alternative of Theorem~\ref{thm:main-dichotomy} would instead imply
\[
u'_{\mu,\ell}(R)\longrightarrow0,
\]
which is incompatible with such a lower bound on $v(R)$. Hence the selected family necessarily realizes the nontrivial, saturated alternative of Theorem~\ref{thm:main-dichotomy}.

Some of the refined shooting estimates require a strengthened version of a technical rotation result already used in \cite{BoCoZi26}. Since this result concerns a fairly general class of planar systems and is independent of the main asymptotic arguments, we present it separately in Appendix~\ref{appendix1}.

The plan of the paper is as follows. In Section \ref{sec2} we recall the radial shooting formulation and the phase-plane notation. Section \ref{sec-dic} is devoted to the proof of the asymptotic dichotomy in Theorem~\ref{thm:main-dichotomy}, while in Section \ref{sec4} we prove Theorem \ref{thm:realization} by constructing the small and large nodal families. Finally, Section \ref{sec:num} collects some numerical simulations that illustrate both asymptotic behaviours and Appendix~\ref{appendix1} contains the technical rotation estimates used in the shooting arguments.

\medbreak
\noindent
\textbf{Acknowledgments.} The authors are very grateful to Andrea Tellini for his generous help with the numerical simulations. His explanations and suggestions were invaluable in clarifying how to perform the computations in a mathematically reliable way, and greatly improved our understanding of the numerical aspects of the problem. 
This work was carried out under the auspices of GNAMPA–INdAM.

\medbreak
\noindent
\textbf{Declaration on the use of AI tools.}
The mathematical ideas, strategies, results, and proofs presented in this work were developed entirely by the authors. ChatGPT 5.6 (OpenAI) was used as an auxiliary tool to assist in checking the presentation and internal consistency of some of the proofs, and in the linguistic and stylistic revision of the manuscript. All mathematical arguments and all AI-assisted suggestions were independently checked and validated by the authors, who take full responsibility for the content of the paper.

\section{Preliminaries}\label{sec2}

In this section, we collect the notation and the main tools that will be used throughout the paper. Most of the material is taken from our previous work \cite{BoCoZi26}, where existence and multiplicity results for radial solutions of \eqref{eq:main} were established using a shooting method.

\subsection{The radial shooting problem}

As in the Introduction, throughout the paper, we denote by $\mathcal B_R \subset \mathbb{R}^N$ the ball centered at the origin with radius $R >0$, where $N \geq 2$. We consider radial solutions of
\begin{equation}\label{eq:PDE}
-\operatorname{div}\left(
\frac{\nabla u}{\sqrt{1-|\nabla u|^2}}
\right)
=\lambda u+\mu g(u),
\qquad
u=0 \quad\text{on }\partial \mathcal{B}_R.
\end{equation}
Writing $u=u(r)$, problem \eqref{eq:PDE} reduces to
\begin{equation}\label{eq:radial}
\begin{cases}
-\displaystyle\left(
r^{N-1}\varphi(u')
\right)'
=
r^{N-1}\bigl(\lambda u+\mu g(u)\bigr),
\qquad r\in(0,R),
\\[2mm]
u'(0)=0,\qquad
u(R)=0,
\end{cases}
\end{equation}
where
$$
\varphi(s)=\frac{s}{\sqrt{1-s^2}},
\qquad |s|<1.
$$

Following the shooting approach introduced in \cite{BoCoZi26}, for every initial value
$\alpha\in\mathbb R$, we consider the Cauchy problem

\begin{equation}\label{system}
\begin{cases}
\begin{aligned}
u' &= \varphi^{-1}\!\left(\frac{v}{r^{N-1}}\right), && r \in (0, R),\\
v' &= -r^{N-1} ( \lambda u+ \mu g(u)), && r \in (0, R),\\
u(0) &= \alpha,\\
v(0) &= 0,
\end{aligned}
\end{cases}
\end{equation}
Under assumptions $(g_1)$ and $(g_2)$, by the Cauchy-Lipschitz theory for ODEs, problem
\eqref{system} admits a unique global solution
$(u_\alpha,v_\alpha)$ depending continuously on the initial datum $\alpha$. More precisely, the following result holds true.

\begin{lemma}[{\cite[Lemma 2.1, Remark 1]{BoCoNo20b}}]
\label{continuityalpha}
For any $\alpha \in \mathbb{R}$, problem \eqref{system} has a unique solution $(u_\alpha, v_\alpha)$ defined on $[0,R]$, which in particular satisfies $$ \lim_{r \to 0^+} \frac{v_\alpha (r)}{r^{N-1}} = 0,$$
that is, $v_\alpha(r) / r^{N-1}$ can always be extended continuously by $0$ at $r=0$.
Moreover, if $\alpha \to \bar{\alpha}$ , then
\begin{equation}
\label{unifalpha}
\begin{aligned}
    &u_{\alpha}(r) \to u_{\bar{\alpha}} (r) \quad &&\text{uniformly on } [0,R], \\
    &v_{\alpha}(r) \to v_{\bar{\alpha}} (r) \quad &&\text{uniformly on } [0,R],\\
    &\frac{v_\alpha(r)}{r^{N-1}} \to \frac{v_{\bar \alpha}(r)}{r^{N-1}} &&\text{uniformly on }[0,R].
\end{aligned}\end{equation}
\end{lemma}
Moreover, by structure, one has that for every $u_\alpha$ solution of \eqref{system}, it holds that
$$
|u_\alpha'(r)|<1,
\qquad
r\in[0,R].
$$
Radial solutions of \eqref{eq:radial} are therefore obtained by selecting
those values of $\alpha$ satisfying
$$
u_\alpha(R)=0.
$$

\subsubsection{Polar coordinates}
\label{subsec:polarcoordinates}
Throughout the paper, we encode the nodal properties of the shooting
trajectory by means of suitable clockwise polar coordinates in the phase plane. 

Let $\alpha>0$, and let $(u_\alpha,v_\alpha)$ be the solution of \eqref{system}. Since $u_\alpha(0)=\alpha>0$, the trajectory does not start from the origin; hence, by uniqueness, $(u_\alpha(r), v_\alpha(r)) \neq (0,0)$ for every $r \in [0,R]$. We may therefore write
\[
u_\alpha(r)=\rho_\alpha(r)\cos\theta_\alpha(r),
 \qquad
 v_\alpha(r)=-\rho_\alpha(r)\sin\theta_\alpha(r),
\]
where $\rho_\alpha(r)>0$ and the continuous angular coordinate is normalized by
\[
 \theta_\alpha(0)=0.
\]
The minus sign in the second identity means that increasing values of $\theta_\alpha$ correspond to clockwise rotation.

As observed in \cite{BoCoZi26}, when $\lambda + \mu > 0$ assumption $(g_3)$ yields
\begin{equation} \label{thetaincreasing} \theta_\alpha'(r) > 0 \quad \text{for every } r\in (0,R],\end{equation}
namely, the trajectory $(u_\alpha,v_\alpha)$ moves clockwise in the phase plane. In particular, between any two consecutive zeros of $u_\alpha$ there is exactly one zero of $u_\alpha'$, and vice versa.
This allows us to characterize the nodal
properties of $u_\alpha$ in terms of the total angular variation
$\theta_\alpha(R)$. More precisely, if
\begin{equation}\label{nodaltheta}
\theta_\alpha(R)=\frac{\pi}{2}+(\ell-1)\pi,
\qquad \ell\ge1,
\end{equation}
then $u_\alpha(R)=0$, so that $u_\alpha$ is a solution of
\eqref{eq:radial}, having exactly $\ell-1$ zeros in $(0,R)$ and exactly one critical point between two consecutive zeros. 

We shall occasionally use rescaled polar coordinates $(\tilde\rho_\alpha,\tilde\theta_\alpha)$. More precisely, for two positive constants $\beta_1,\beta_2>0$, not simultaneously equal to 1, we may write
\begin{equation}\label{scaledpolar}
u_\alpha(r)=\beta_1\tilde\rho_\alpha(r)\cos\tilde\theta_\alpha(r),
    \qquad
    v_\alpha(r)=-\beta_2\tilde\rho_\alpha(r)\sin\tilde\theta_\alpha(r).
\end{equation}
These are the standard clockwise polar coordinates of the rescaled trajectory
\[
 r\longmapsto
(\tilde u_\alpha(r), \tilde v_\alpha(r))= \left(\frac{u_\alpha(r)}{\beta_1},
       \frac{v_\alpha(r)}{\beta_2}\right).
\]
For future use, we highlight here that, since a positive rescaling of the coordinate axes preserves the signs of both components, it preserves each quadrant, each coordinate half-axis, the orientation of the rotation, and all axis-crossing times.
More precisely, since $\theta_\alpha$ and $\tilde\theta_\alpha$ are continuous and normalized by
\[
\theta_\alpha(0)=\tilde\theta_\alpha(0)=0,
\]
for every $i\in\mathbb Z$,
\[
 \theta_\alpha(r)\in
 \left(i\frac{\pi}{2},(i+1)\frac{\pi}{2}\right)
 \quad\Longleftrightarrow\quad
 \tilde\theta_\alpha(r)\in
 \left(i\frac{\pi}{2},(i+1)\frac{\pi}{2}\right),
\]
and
\[
 \theta_\alpha(r)=i\frac{\pi}{2}
 \quad\Longleftrightarrow\quad
 \tilde\theta_\alpha(r)=i\frac{\pi}{2}.
\]
Moreover, a standard computation (cf. \cite[formula (3.8)]{BoCoZi26}) yields in $(0,R)$ 
\begin{equation}
    \label{eqtheta}        \tilde\theta'_\alpha(r) \frac{u'_\alpha v_\alpha - v_\alpha' u _\alpha}{\beta_1 \beta_2 \tilde \rho_\alpha^2} = \frac{\rho_\alpha^2}{\beta_1 \beta_2 \tilde \rho_\alpha^2} \theta'_\alpha >0
\end{equation}
Therefore, also for the rescaled coordinates we can conclude that $\tilde \theta'_\alpha>0$ in $(0,R]$ if $\lambda+\mu>0$. 
It follows that positive rescalings of the axes do not affect the number or the order of the zeros and critical points of $u_\alpha$, nor the rotation count encoded by the successive crossings of the coordinate axes.

\section{Proof of Theorem \ref{thm:main-dichotomy}}
\label{sec-dic}

In this section, we consider problem \eqref{eq:radial} for $\lambda, \mu \geq 0$ and we provide the proof of Theorem \ref{thm:main-dichotomy}. A few comments on the statement and the proof are in order. 

In alternative (1), we expect the convergence $u_\mu' \to 0$ to be uniform on the whole interval $[0,R]$. However, since this fact will only be used as a technical (albeit crucial) ingredient in the arguments of Section \ref{largesols}, we restrict ourselves to proving the weaker statement that $u_\mu' \to 0$ uniformly on $[\bar r,R]$ for some $\bar r \in (0,R)$, which is sufficient for our purposes. The proof of even this weaker result is rather technical.

For this reason, and in order to improve readability, we divide the proof of Theorem \ref{thm:main-dichotomy} into two parts. In the first part, presented immediately below, we establish the dichotomy between alternatives (1) and (2), without addressing the additional convergence property of $u_\mu'$ in case (1). The latter requires several technical results, which will be proved in the remainder of the proof.

\begin{proof}[Proof of Theorem \ref{thm:main-dichotomy} - first part]
Let us first recall that, due to the bound $|u_\mu'| < 1$, it holds that $|u_\mu(r)| \leq R$ for every $r \in [0,R]$ and for every $\mu > \bar\mu$. Then, as a direct consequence of the Arzelà-Ascoli theorem, we have, up to subsequence, $u_\mu \to u_\infty$ uniformly, with $u_\infty$ a Lipschitz continuous function with $|u_\infty'(r)| \leq 1$ for a.e. $r \in [0,R]$. Moreover, using again the bound $|u_\mu'| < 1$, the sequential version of the Banach-Alaoglu theorem yields $u_\mu' \to u_\infty'$ in the weak-$*$ topology of $L^\infty$.

On the other hand, recalling $v_\mu(r)=r^{N-1}\varphi(u_\mu'(r))$ as in \eqref{system}, let us define \begin{equation} \label{def:w}w_\mu(r):= \frac{v_\mu(r)}{\mu}=-\int_0^r s^{N-1} \left(\frac{\lambda}{\mu} u_\mu(s) + g(u_\mu(s)) \right) ds.\end{equation}
Using the uniform boundedness of $u_\mu$, it is easily seen that both $w_\mu$ and $w_\mu'$ are uniformly bounded, as well. Therefore, using the Arzelà--Ascoli theorem again, we have that, up to a further subsequence, $w_\mu \to w_\infty$ uniformly in $[0,R]$. From \eqref{def:w} and the uniform convergence of $\{u_\mu\}_{\mu\geq \bar \mu}$ we deduce that, in fact,
    $$w_\infty(r) =- \int_0^r s^{N-1} g(u_\infty(s)) ds,$$
    and so $w_\infty \in C^1([0,R])$. Moreover, by the definition of $v_\mu$ we have for $r \in (0,R]$
    $$\varphi(u_\mu'(r)) = \mu \frac{w_\mu(r)}{r^{N-1}},$$
    which, for $\mu \to +\infty$, goes pointwise to $- \infty$ if $w_\infty(r) < 0$, and to $+ \infty$ if $w_\infty(r) > 0$. Hence, $u_\mu'(r) \to -1$ and $+1$ respectively, that is, equivalently, $u_\mu'(r) \to \operatorname{sign}(w_\infty(r))$ pointwise on the open set $\{w_\infty(r) \neq 0\}$. By uniqueness of the limit, it follows that
    $$u_\infty'=\operatorname{sign}(w_\infty) \text{ a.e. in } \{w_\infty \not = 0\}.$$
Summing up, we have thus proved that the pair $(u_\infty, w_\infty)$ satisfies
    \begin{equation}
    \label{limitproblem}
        \begin{cases}
            \begin{aligned}
                u_\infty' &= \operatorname{sign}(w_\infty) &&\text{a.e. in } \{w_\infty \not = 0\}, \\
                w_\infty' &= -r^{N-1} g(u_\infty) && r\in (0,R).
            \end{aligned}
        \end{cases}
    \end{equation}
    
Of course, the pair $(u_\infty, w_\infty) \equiv (0,0)$ is a solution of \eqref{limitproblem}: in this case, we are in the first alternative given in the statement of the theorem and what has still to be proved is that $u_\mu' \to 0$ uniformly on $[\bar{r},R]$ for some $\bar r \in (0,R)$. This is rather technical, and it will be postponed to the second part of the proof.

On the other hand, let us assume that $u_\infty \not\equiv 0$. In this case, we first claim that
\begin{equation}\label{non-zero}
u_\infty^2(r) + w^2_\infty(r) \not = 0 \quad \text{for every } r \in [0,R].
\end{equation}
Suppose by contradiction that
$$
u_\infty(r_0)=w_\infty(r_0)=0
$$
for some $r_0\in[0,R]$. Let
$$
Z:=\{r\in[0,R]:(u_\infty(r),w_\infty(r))=(0,0)\},
$$
and let $I$ be the connected component of $Z$ containing $r_0$. Since the pair $(u_\infty,w_\infty)$ is not identically zero, $I\neq[0,R]$. Hence at least one endpoint of $I$ has points of $[0,R]\setminus Z$ arbitrarily close to it. Assume, without loss of generality, that this happens at the right endpoint $a<R$ of $I$. Thus
$$
u_\infty(a)=w_\infty(a)=0,
$$
while $(u_\infty,w_\infty)$ is not identically zero on any interval $[a,a+\varepsilon]$. We claim that $u_\infty$ takes both positive and negative values in every interval $(a,a+\varepsilon)$. Suppose otherwise. Then, for some $b>a$, either
$$
u_\infty\geq0\quad\text{on }[a,b] \quad 
\text{or} \quad 
u_\infty\leq0\quad\text{on }[a,b].
$$
We consider the first case, the second one being completely analogous. By assumption $(g_3)$,
$$
g(u_\infty)\geq0\quad\text{on }[a,b],
$$
and hence, by \eqref{limitproblem},
$$
w_\infty'\leq0\quad\text{on }[a,b].
$$
Since $w_\infty(a)=0$, from the latter we deduce $w_\infty\leq0$ on $[a,b]$. We show that in fact $w_\infty\equiv0$ on $[a,b]$. Otherwise, let $(c,d)$ be a connected component of the open set
$$
\{r\in(a,b):w_\infty(r)<0\}.
$$
Assume first that $c> a$: by continuity, $w_\infty(c)=0$. Moreover, since $c \in (a,b)$, the fact that $w_\infty\leq0$ implies that $c$ is a local maximum of $w_\infty$, and hence $w_\infty'(c)=0$.  By \eqref{limitproblem} and $(g_3)$, this gives $u_\infty(c)=0$. On the other hand, as stated before, on $(c,d)$ we have $w_\infty<0$, and therefore
$$
u_\infty'=-1
\qquad\text{a.e. on }(c,d).
$$
It follows that
$$
u_\infty(r)=-(r-c)<0
\qquad\text{for }r\in(c,d),
$$
contradicting $u_\infty\geq0$ on $[a,b]$. Now, assume that $c= a$: In this case, since we have $u_\infty' =-1$ a.e. in $(a,d)$, we deduce 
$$u_\infty(r) = -(r-a) < 0 \quad \text{in } (a,d),$$
also contradicting that $u_\infty \geq 0$ in $[a,b]$. Therefore the set $\{ r \in (a,b) : w_\infty(r) < 0\}$ is empty, and hence $w_\infty\equiv 0$ on $[a,b]$. Then $w_\infty'\equiv0$, and the second equation in \eqref{limitproblem}, together with $(g_3)$, yields
$$
u_\infty\equiv0\quad\text{on }[a,b],
$$
contradicting the choice of $a$.

In conclusion, $u_\infty$ takes both positive and negative values in every right neighbourhood of $a$. We may consequently choose $\ell+1$ points
$$
a<r_1<r_2<\cdots<r_{\ell+1}<R
$$
such that
$$
u_\infty(r_j)u_\infty(r_{j+1})<0,
\qquad j=1,\ldots,\ell.
$$
Indeed, such a finite sequence can be extracted from a strictly decreasing sequence of points converging to $a$ from the right, where $u_\infty$ assumes alternating signs, and then conveniently relabeled. Since $u_\mu\to u_\infty$ uniformly, for all sufficiently large $\mu$ we have simultaneously
$$
u_\mu(r_j)u_\mu(r_{j+1})<0,
\qquad j=1,\ldots,\ell.
$$
Hence $u_\mu$ has at least one zero in each interval
$$
(r_j,r_{j+1}),\qquad j=1,\ldots,\ell,
$$
and thus at least $\ell$ zeros in $(0,R)$. This contradicts the assumption that $u_\mu$ has exactly $\ell-1$ zeros. The case in which the relevant endpoint of $I$ is approached from the left is completely symmetric. This proves \eqref{non-zero}.

As a consequence of \eqref{limitproblem} and \eqref{non-zero},
$w_\infty$ has only finitely many zeros. Indeed, since $u_\mu(0)>0$
and $u_\mu\to u_\infty$ uniformly, we have $u_\infty(0)\geq0$.
Moreover, $w_\infty(0)=0$ and \eqref{non-zero} imply
$u_\infty(0)\neq0$, and hence $u_\infty(0)>0$. Therefore, by
continuity, $u_\infty(r)>0$ for $r>0$ sufficiently small, and thus,
by \eqref{limitproblem} and $(g_3)$, $w_\infty$ is strictly
decreasing on a right neighborhood of zero. Hence $r=0$ is an
isolated zero of $w_\infty$. Moreover, since $u_\infty(R)=0$,
\eqref{non-zero} gives $w_\infty(R)\neq0$. Thus, if $w_\infty$ had
infinitely many zeros, they would admit an accumulation point
$\bar r\in(0,R)$. Since $w_\infty\in C^1([0,R])$, we would have
$w_\infty(\bar r)=w_\infty'(\bar r)=0.$
Then, from the second equation in \eqref{limitproblem} and $(g_3)$,
we would obtain $u_\infty(\bar r)=0$, contradicting
\eqref{non-zero}.

Now, again using \eqref{non-zero}
we can write the planar path $(u_\infty,w_\infty)$ in polar coordinates as
$$
\begin{cases}
u_\infty(r) = \rho_\infty(r) \cos\theta_\infty(r) \\
w_\infty(r) = -\rho_\infty(r) \sin \theta_\infty(r),
\end{cases}
$$
where $\rho_\infty(r) > 0$ and $\theta_\infty(0) = 0$. Moreover, $\theta_\infty$ is differentiable outside the (finite) set of zeros of $w_\infty$, with
\begin{equation} \label{theta-derivative} \begin{aligned}
    \theta_\infty'(r) &= \left (\frac{u_\infty'(r) w_\infty(r) - w_\infty'(r) u_\infty(r)}{u_\infty^2(r) + w_\infty^2(r)} \right) \\
    &= \left( \frac{ \operatorname{sign}(w_\infty(r) ) w_\infty(r) + r^{N-1} g(u_\infty(r)) u_\infty(r)}{u_\infty^{2} (r) + w_\infty^2 (r) }\right) > 0.\end{aligned}
\end{equation}

Now, since $(u_\mu,w_\mu)$ converges uniformly to $(u_\infty,w_\infty)$
and, by \eqref{non-zero}, the limiting path
$(u_\infty,w_\infty)$ does not meet the origin, for $\mu$ sufficiently
large the paths $(u_\mu,w_\mu)$ and $(u_\infty,w_\infty)$ are homotopic
in $\mathbb{R}^2\setminus\{0\}$ through the straight-line homotopy.
Moreover, their endpoints lie on the same rays, and therefore they have
the same angular variation. Hence
\[
\theta_\infty(R)-\theta_\infty(0)
=
\Delta\arg(u_\mu,w_\mu),
\]
where by $\Delta \arg(\cdot, \cdot)$ we mean the clockwise angular variation. On the other hand, the paths $(u_\mu,v_\mu)$ and $(u_\mu,w_\mu)$ also
have the same angular variation. Indeed, since $w_\mu=v_\mu/\mu$ and
$\mu>0$, they are joined in $\mathbb{R}^2\setminus\{0\}$ by the homotopy
\[
H_\mu(s,r)
=
\left(
u_\mu(r),
\left((1-s)+\frac{s}{\mu}\right)v_\mu(r)
\right),
\qquad (s,r)\in[0,1]\times[0,R].
\]
The endpoints remain on the same rays throughout this homotopy.
Consequently,
\[
\theta_\infty(R)-\theta_\infty(0)
=
\Delta\arg(u_\mu,v_\mu)
=
\frac{\pi}{2}+(\ell-1)\pi.
\]
Since $\theta_\infty(0)=0$, we conclude that
\begin{equation} \label{total-angle}
\theta_\infty(R)
=
\frac{\pi}{2}+(\ell-1)\pi.
\end{equation}

Since by \eqref{theta-derivative} the path moves clockwise, from
\eqref{limitproblem} and \eqref{total-angle} one deduces that
$u_\infty$ has exactly $\ell-1$ zeros in $(0,R)$, is strictly
decreasing from $r=0$ to its first zero, and has exactly one local
extremum between any two consecutive zeros. Finally, since
$w_\infty$ has only finitely many zeros, the set
$\{w_\infty\neq0\}$ has finitely many connected components. On each
of them, $w_\infty$ has constant sign and, by \eqref{limitproblem},
\[
u_\infty'=\operatorname{sign}(w_\infty)=\pm1
\qquad\text{a.e.}
\]
Hence $u_\infty$ is affine on each such component, and therefore
piecewise affine on $[0,R]$, with $|u_\infty'|=1$ almost everywhere.
This concludes the first part of the proof.
\end{proof}
As already mentioned, the rest of the section will be devoted to proving the additional information concerning the convergence $u_\mu' \to 0$ in the case (1). For this, we first need to define, following the ideas in \cite{CaDaZa01}, a notion of energy along a solution $u_\mu$ as
\begin{equation}
\label{energy}
E_\mu(r) 
= \frac{1}{\sqrt{1-u_\mu'(r)^2}} - 1 + \frac{\lambda}{2} u_\mu(r)^2 + \mu G(u_\mu(r)),
\end{equation}
where $G(s) = \int_0^s g(\xi)\,d\xi$.
Note that $E_\mu(r) \geq 0$ for every $r \in [0,R]$; moreover, we claim that
\begin{lemma}
\label{decre}
    $E_\mu(r)$ is non-increasing in $[0,R]$.
\end{lemma}
\begin{proof}[Proof of the lemma]
 Indeed, differentiating with respect to $r$ in \eqref{energy} we obtain
$$
E_\mu'(r)
= -\frac{1}{2}
\frac{-2u_\mu'(r) u_\mu''(r) }{(1-u_\mu'(r)^2)^{3/2}}
+ \left[\lambda u_\mu(r) + \mu g(u_\mu(r))\right]u_\mu'(r),
$$
that is
$$
E_\mu'(r)
= u_\mu'(r) 
\left[
\frac{u_\mu''(r) }{(1-u_\mu'(r)^2)^{3/2}}
+ \lambda u_\mu(r) + \mu g(u_\mu(r))
\right].
$$
Replacing in the above formula the expression for $u_\mu''$ coming from the equation, that is,
$$-u_\mu''(r) = (\lambda u_\mu(r) + \mu g(u_\mu(r)))(1-u_\mu'(r)^2)^{3/2} + \frac{(N-1)}{r} u_\mu'(r)(1-u_\mu'(r)^2),$$
we finally obtain
\begin{equation}
\label{denergy}
E_\mu'(r)
= -\frac{N-1}{r}
\frac{u_\mu'(r)^2}{\sqrt{1-u_\mu'(r)^2}}
\le 0,
\end{equation}
which concludes the proof. \end{proof}

With this result, we can proceed with the second part of the proof of Theorem \ref{thm:main-dichotomy}.
\begin{proof}[Proof of Theorem \ref{thm:main-dichotomy} - second part.] As said before, the rest of the proof will be devoted to proving the additional information concerning the convergence $u_\mu' \to 0$ in the case (1). Thus, we have to prove that, if $u_\infty \equiv 0$, then $u_\mu' \to 0$ uniformly on $[\bar r,R]$ for some $\bar r \in (0,R).$ We first observe that this follows if we show that
\begin{equation}\label{stimagmu}
\mu \Vert g(u_\mu) \Vert_{L^\infty(\bar r,R)} = O(1)
 \qquad \mu\to+\infty.
\end{equation}
Indeed, from equation \eqref{eq:radial} we obtain
for $r \in (0, R]$,
\begin{equation}\label{formulau''}
-u_\mu''(r) = (\lambda u_\mu(r) + \mu g(u_\mu(r)))(1-u_\mu'(r)^2)^{3/2} + \frac{(N-1)}{r} u_\mu'(r)(1-u_\mu'(r)^2)
\end{equation}
and so from the bound \eqref{stimagmu} we immediately get
$$
\Vert u_\mu'' \Vert _{L^\infty(\bar r,R)} = O(1)
  \quad\text{as } \mu\to+\infty.
$$
Then, by the Arzelà--Ascoli theorem, $u_\mu'$ converges uniformly on $[\bar r, R]$ to a continuous function which, since $u_\infty\equiv 0$, must be the trivial one.

So, let us prove \eqref{stimagmu}. Recalling the discussion in Section \ref{subsec:polarcoordinates}, we can order the zeros $\{r_\mu^i\}_{i=1}^\ell$ and the critical points $\{s_\mu^i\}_{i=0}^{\ell-1}$ of $u_\mu$ as follows
\[
  0 = s^0_\mu < r^1_\mu < s^1_\mu < r^2_\mu < \cdots
  < r^{\ell-1}_\mu < s^{\ell-1}_\mu < r^\ell_\mu = R,
\]
thus producing $2\ell-1$ consecutive intervals whose lengths sum to $R$. By the pigeonhole principle, we deduce that, for every $\mu$, there exists at least one interval of the type $[r_\mu^i,s_\mu^i]$ or $[s_\mu^i,r_\mu^{i+1}]$ having length greater than or equal to $R/(2\ell-1)$. Hence, for every $\mu$ there exists an index $i=i(\mu)$ such that
\begin{equation}
\label{liminf}
r_\mu^{i+1}-s_\mu^i \geq \frac{R}{2\ell-1}
\quad\text{or}\quad
s_\mu^i-r_\mu^i \geq \frac{R}{2\ell-1}.
\end{equation}
Set
\[
c_0:=\frac{R}{2(2\ell-1)}
\]
and fix $\delta_0\in(0,c_0/4)$. For every $\mu$, we select one of the intervals satisfying \eqref{liminf} and define
\[
J_\mu=[a_\mu,b_\mu]
\]
by setting
\[
J_\mu=[s_\mu^i+\delta_0,r_\mu^{i+1}-\delta_0]
\]
if the first alternative in \eqref{liminf} holds, and
\[
J_\mu=[r_\mu^i+\delta_0,s_\mu^i-\delta_0]
\]
if the second one holds. Notice that, in both cases,
\begin{equation}
\label{Jlength}
b_\mu-a_\mu\geq 2c_0-2\delta_0>c_0
\end{equation}
and $a_\mu\geq\delta_0$. Moreover, $u_\mu$ has constant sign on the nodal interval containing $J_\mu$. We denote such a sign by $\varepsilon_\mu\in\{-1,1\}$, so that
\[
\varepsilon_\mu u_\mu>0
\]
on $J_\mu$. In this way, the four possible configurations, corresponding to the two possible signs of $u_\mu$ and to the two possible alternatives in \eqref{liminf}, can be treated at the same time. Importantly, all the constants appearing in the estimates below can be chosen independently of the particular index $i=i(\mu)$ and of the interval selected in \eqref{liminf}. From now on, in order to lighten the notation, we will simply write $i$ in place of $i(\mu)$.

With this in mind, we first prove two technical claims. 

\begin{claim}
\label{claim1}
There exists $K_1>0$, independent of $\mu$, such that, for every $\mu$,
\begin{equation}
\label{claim1:largesols}
|u_\mu'(r)|\leq K_1|u_\mu(r)|,
\quad \text{for every }r\in J_\mu.
\end{equation}
\end{claim}

\begin{proof}[Proof of the Claim]
We distinguish the two possible alternatives in \eqref{liminf}.

\noindent
Assume first that the selected interval is $[s_\mu^i,r_\mu^{i+1}]$. In this case,
\[
\varepsilon_\mu u_\mu'<0
\qquad\text{in }(s_\mu^i,r_\mu^{i+1}).
\]
Since $\varphi$ is an increasing odd function, and from the equation, the map $r\mapsto -\varepsilon_\mu r^{N-1}\varphi(u_\mu'(r)) =r^{N-1}\varphi(|u_\mu'(r)|)$ is positive and increasing on this interval. Therefore, fixing any
$r\in J_\mu$, for every $\xi\in[r,r_\mu^{i+1}]$ we have
\[
\xi^{N-1}\varphi(|u_\mu'(\xi)|)
\geq
r^{N-1}\varphi(|u_\mu'(r)|).
\]
Applying $\varphi^{-1}$ to both sides and using the elementary convexity inequality
\[
\varphi^{-1}(\theta\varphi(s))\geq\theta s,
\quad \text{for every }s\in[0,1)\text{ and }\theta\in[0,1],
\]
together with $\delta_0\leq r\leq\xi\leq r_\mu^{i+1}\leq R$, we obtain
\[
|u_\mu'(\xi)|
\geq
\left(\frac{r}{\xi}\right)^{N-1}|u_\mu'(r)|
\geq
\left(\frac{\delta_0}{R}\right)^{N-1}|u_\mu'(r)|.
\]
Now, using $u_\mu(r_\mu^{i+1})=0$ and integrating from $r$ to $r_\mu^{i+1}$, we get
\[
\begin{aligned}
|u_\mu(r)|
&=
\int_r^{r_\mu^{i+1}}|u_\mu'(\xi)|\,d\xi \geq
\left(\frac{\delta_0}{R}\right)^{N-1}
(r_\mu^{i+1}-r)|u_\mu'(r)|\geq
\left(\frac{\delta_0}{R}\right)^{N-1}
\delta_0|u_\mu'(r)|.
\end{aligned}
\]
Therefore,
\[
|u_\mu'(r)|
\leq
\frac{R^{N-1}}{\delta_0^N}|u_\mu(r)|.
\]

Assume now that the selected interval is $[r_\mu^i,s_\mu^i]$. In this case,
\[
\varepsilon_\mu u_\mu'>0
\qquad\text{in }(r_\mu^i,s_\mu^i).
\]
Recalling \eqref{formulau''}, we have
\[
-u_\mu''
=
(\lambda u_\mu+\mu g(u_\mu))(1-(u_\mu')^2)^{3/2}
+\frac{N-1}{r}u_\mu'(1-(u_\mu')^2).
\]
Both terms on the right-hand side have sign $\varepsilon_\mu$. Hence,
\[
\varepsilon_\mu u_\mu''<0
\]
on $(r_\mu^i,s_\mu^i)$ and therefore $|u_\mu'|$ is decreasing on this interval. Thus, for every $r\in J_\mu$ and every $\xi\in[r_\mu^i,r]$,
\[
|u_\mu'(\xi)|\geq |u_\mu'(r)|.
\]
Using $u_\mu(r_\mu^i)=0$, we obtain
\[
\begin{aligned}
|u_\mu(r)|
&=
\int_{r_\mu^i}^r|u_\mu'(\xi)|\,d\xi \geq
(r-r_\mu^i)|u_\mu'(r)| \geq
\delta_0|u_\mu'(r)|.
\end{aligned}
\]
The conclusion follows in both cases, by taking, for instance,
\[
K_1=
\max\left\{
\frac{R^{N-1}}{\delta_0^N},
\frac1{\delta_0}
\right\}.
\]
\end{proof}

\begin{claim}
\label{claim2}
There exists $K_2>1$, independent of $\mu$, such that, for every $\mu$,
\begin{equation}
\label{claim2:largesols}
K_2^{-1}|u_\mu(a_\mu)|
\leq |u_\mu(r)|
\leq K_2|u_\mu(a_\mu)|,
\quad \text{for every }r\in J_\mu.
\end{equation}
\end{claim}

\begin{proof}[Proof of the Claim]
By Claim \ref{claim1}, one has
\[
\frac{|u_\mu'(r)|}{|u_\mu(r)|}\leq K_1,
\quad \text{for every }r\in J_\mu,
\]
or, equivalently,
\[
\left|\frac{d}{dr}\log|u_\mu(r)|\right|
\leq K_1,
\quad \text{for every }r\in J_\mu.
\]
Integrating this from $a_\mu$ to $r$, and using that $|r-a_\mu|\leq R$, it yields
\[
e^{-K_1R}|u_\mu(a_\mu)|
\leq
|u_\mu(r)|
\leq
e^{K_1R}|u_\mu(a_\mu)|,
\]
from which the conclusion follows.
\end{proof}

Now we can finally prove \eqref{stimagmu}. To begin, given Claims~\ref{claim1} and~\ref{claim2}, there exist constants $K_1>0$ and $K_2>1$, independent of $\mu$, such that
\begin{align}
|u_\mu'(r)|
&\leq K_1|u_\mu(r)|
&&\text{for all \,}r\in J_\mu,
\label{eq:C1bound}\\
K_2^{-1}|u_\mu(a_\mu)|
\leq |u_\mu(r)|
&\leq K_2|u_\mu(a_\mu)|
&&\text{for all \,}r\in J_\mu.
\label{eq:Kbound}
\end{align}
 From now on, $C$ will refer to a generic constant independent of $\mu$, that might change from line to line. Note that, since $u_\mu\to0$ uniformly on $[0,R]$, for $\mu$ large enough one has
\[
K_1\|u_\mu\|_{L^\infty(0,R)}\leq\frac12,
\]
which, combined with \eqref{eq:C1bound}, yields
\begin{equation}
\label{1half}
|u_\mu'(r)|\leq\frac12
\quad\text{in }J_\mu.
\end{equation}
Then, using the elementary estimate
\[
|\varphi(s)|\leq 2\varphi(1/2)|s|,
\qquad |s|\leq\frac12,
\]
together with \eqref{eq:C1bound} and \eqref{eq:Kbound}, we obtain, for every $r\in J_\mu$,
\begin{equation}
\label{eq:phibound}
|\varphi(u_\mu'(r))|
\leq
2\varphi(1/2)|u_\mu'(r)|
\leq
C|u_\mu(a_\mu)|,
\end{equation}
Integrating \eqref{eq:radial} over $J_\mu=[a_\mu,b_\mu]$, we obtain
\begin{equation}
\label{eq:integrated}
\int_{a_\mu}^{b_\mu}
r^{N-1}\bigl(\lambda u_\mu+\mu g(u_\mu)\bigr)\,dr
=
-\Bigl[r^{N-1}\varphi(u_\mu')\Bigr]_{a_\mu}^{b_\mu}.
\end{equation}
Multiplying the previous equality by $\varepsilon_\mu$, and recalling that $u_\mu$ and $g(u_\mu)$ have the same sign on $J_\mu$, gives
\[
\int_{a_\mu}^{b_\mu}
r^{N-1}\bigl(\lambda|u_\mu|+\mu|g(u_\mu)|\bigr)\,dr
=
-\varepsilon_\mu
\Bigl[r^{N-1}\varphi(u_\mu')\Bigr]_{a_\mu}^{b_\mu}.
\]
Using \eqref{eq:phibound} on the right-hand side, we obtain
\begin{equation}
\label{eq:phibound2}
\int_{a_\mu}^{b_\mu}
r^{N-1}\bigl(\lambda|u_\mu|+\mu|g(u_\mu)|\bigr)\,dr
\leq
C|u_\mu(a_\mu)|.
\end{equation}

On the other hand, for $\mu$ sufficiently large, so that $|u_\mu|<\delta$, assumption $(g_2)$ and \eqref{eq:Kbound} imply
\[
|g(u_\mu(r))|
\geq
C|u_\mu(a_\mu)|^{\gamma-1},
\qquad
r\in J_\mu,
\]
Moreover, from \eqref{Jlength} and $a_\mu\geq\delta_0$,
\[
\int_{a_\mu}^{b_\mu}r^{N-1}\,dr
\geq a_\mu^{N-1} (b_\mu - a_\mu) \geq \delta_0^{N-1} c_0 \geq C,
\]
Therefore,
\[
\int_{a_\mu}^{b_\mu}
r^{N-1}\mu|g(u_\mu)|\,dr
\geq
\mu C
|u_\mu(a_\mu)|^{\gamma-1}.
\]
Combining this estimate with \eqref{eq:phibound2}, we conclude that
\begin{equation}
\label{gamma2}
\mu|u_\mu(a_\mu)|^{\gamma-2}\leq C.
\end{equation}

\medskip

Recalling \eqref{1half}, one has in particular that
$|u_\mu'(a_\mu)|\leq\tfrac12$. Together with Claim~\ref{claim1}, this allows us to obtain
\[
\frac1{\sqrt{1-u_\mu'(a_\mu)^2}}-1
\leq
C|u_\mu'(a_\mu)|^2
\leq
C|u_\mu(a_\mu)|^2.
\]
Moreover, by $(g_2)$ and up to taking $\mu$ large enough, we also have
\[
\begin{aligned}
\mu G(u_\mu(a_\mu))
&\leq
C\mu|u_\mu(a_\mu)|^\gamma =
C|u_\mu(a_\mu)|^2
\bigl(\mu|u_\mu(a_\mu)|^{\gamma-2}\bigr)
\leq
C|u_\mu(a_\mu)|^2,
\end{aligned}
\]
where we use \eqref{gamma2} for the passage from the second to the third line. Hence, evaluating the energy $E_\mu$ (see \eqref{energy}) in $a_\mu$ and using the above estimates, we obtain
\[
E_\mu(a_\mu)
\leq
C|u_\mu(a_\mu)|^2.
\]
Finally, since $E_\mu$ is non-increasing (see Lemma \ref{decre}), for every
$r\in[a_\mu,R]$ we have
\[
\mu G(u_\mu(r))
\leq
E_\mu(r)
\leq
E_\mu(a_\mu)
\leq
C|u_\mu(a_\mu)|^2.
\]
Using again $(g_2)$, we conclude that
\begin{equation}
\label{u-energy-bound}
\mu|u_\mu(r)|^\gamma
\leq
C|u_\mu(a_\mu)|^2,
\qquad
r\in[a_\mu,R].
\end{equation}

It remains only to choose an interval independent of $\mu$. Indeed, since the interval selected in \eqref{liminf} has length at least $2c_0$, its left endpoint is smaller than or equal to $R-2c_0$. Therefore,
\[
a_\mu\leq R-2c_0+\delta_0<R-\frac{c_0}{2}.
\]
Hence, setting
\[
\bar r:=R-\frac{c_0}{2}\in(0,R),
\]
we have
\[
[\bar r,R]\subset[a_\mu,R]
\]
for every $\mu$. Finally, from \eqref{gamma2},
\[
|u_\mu(a_\mu)|
\leq
C\mu^{-1/(\gamma-2)}.
\]
Combining this with \eqref{u-energy-bound}, we obtain
\[
|u_\mu(r)|
\leq
C\mu^{-1/(\gamma-2)},
\qquad
r\in[\bar r,R].
\]
Using once more $(g_2)$, we conclude that
\[
\mu|g(u_\mu(r))|
\leq
C\mu|u_\mu(r)|^{\gamma-1}
\leq
C\mu^{-1/(\gamma-2)}
\leq C,
\qquad
r\in[\bar r,R].
\]
This proves \eqref{stimagmu}, hence concluding the proof.

\end{proof}

\section{Proof of Theorem \ref{thm:realization}}\label{sec4}

Throughout this section, we fix $\lambda \ge 0$ and we work with problem \eqref{eq:radial}. We rely on \cite{BoCoZi26} for
the existence and multiplicity of solutions with a prescribed number of zeros, and we use
the shooting parametrization $(u_\alpha, v_\alpha)$ of \eqref{system} and the associated
polar coordinates of Section~2.1.1. Our goal is to select, among the solutions given by
\cite[Theorem~2.1]{BoCoZi26}, two distinguished families realizing the two alternatives of Theorem \ref{thm:main-dichotomy}.

\subsection{Construction of small solutions}

We first recall two elementary facts concerning the shooting problem.
For $\alpha>0$, let $(u_\alpha,v_\alpha)$ denote the solution of
\eqref{system} with
\[
u_\alpha(0)=\alpha,
\qquad
v_\alpha(0)=0,
\]
and let $\theta_\alpha$ be its clockwise angular coordinate, normalized by
$\theta_\alpha(0)=0$. The first fact concerns the angular behaviour of trajectories starting
sufficiently close to the origin. More precisely, if
\[
\lambda<\lambda_\ell^{\mathrm{rad}},
\]
then, for every fixed $\mu\geq0$, there exists
$\alpha_*=\alpha_*(\lambda,\mu)>0$ such that
\begin{equation}
\label{eq:small-rotation}
\theta_\alpha(R)
<
\frac{\pi}{2}+(\ell-1)\pi
\qquad
\text{for every }\alpha\in(0,\alpha_*).
\end{equation}
This follows from the same comparison argument with the $\ell$-th
radial eigenfunction of the Dirichlet Laplacian used in
\cite[Lemma~2.7]{BoCoZi26}.

For later use, when $\mu>0$, we denote by
$\widetilde\rho_\alpha,\widetilde\theta_\alpha$
the coordinates in \eqref{scaledpolar} corresponding to
$\beta_1=\mu^{-1/\gamma}$ and $\beta_2=1$, with
$\widetilde\theta_\alpha(0)=0$.
Thus,
\[
    u_\alpha(r)=
    \mu^{-1/\gamma}\widetilde\rho_\alpha(r)
    \cos\widetilde\theta_\alpha(r),
    \qquad
    v_\alpha(r)=
    -\widetilde\rho_\alpha(r)
    \sin\widetilde\theta_\alpha(r).
\]
By $(g_3)$, both $\theta_\alpha$ and
$\widetilde\theta_\alpha$ are strictly increasing on $(0,R]$.
Moreover, as recalled in Section~\ref{subsec:polarcoordinates}, they have the same
comparisons with integer multiples of $\pi/2$.

The key technical tool for the construction of the family of small solutions is the following quantitative estimate on the angular displacement. It is essentially \cite[Lemma~2.9]{BoCoZi26}; we include the proof, based on a rescaling argument combined with Proposition~\ref{Prop21}, for the sake of transparency, since we are going to use it in a slightly different way than in the original work. 

\begin{lemma}
\label{claimellipse}
Let $\lambda\geq 0$. There exists $i_0\in\mathbb{N}$ such that, for every integer $i\geq i_0$, there exist $
\widetilde\mu_i=\widetilde\mu_i(\lambda)\geq 1$, $\Gamma_i=\Gamma_i(\lambda)\in(0, 1/i),
$ with
$$
\widetilde\mu_i\to+\infty,
\qquad
\Gamma_i\to0
\qquad\text{as }i\to+\infty,
$$
such that, for every $\mu\geq\widetilde\mu_i$, every solution $(u(r),v(r))$ of the differential system in \eqref{system}, defined on $[1/i,R]$ and satisfying
\begin{equation}\label{eq:threshold}
\mu^{2/\gamma}u(1/i)^2+v(1/i)^2=\Gamma_i^2,    
\end{equation}
fulfills
$$
\widetilde\theta(R)-\widetilde\theta(1/i)>(i+1)\pi,
$$
where $\gamma$ is as in $(g_2)$ and
$\widetilde\theta$ is the scaled angular coordinate
in \eqref{scaledpolar}, with
$\beta_1=\mu^{-1/\gamma}$ and $\beta_2=1$. Moreover, the sequence $\{\widetilde\mu_i\}_{i\geq i_0}$ can be chosen strictly increasing.
\end{lemma}

\begin{proof}
Choose $i_0\in\mathbb{N}$ sufficiently large so that $1/i<R$ for every $i\geq i_0$, and fix such an $i$. We apply Proposition~\ref{Prop21} after a suitable rescaling. Set
$$
r=\frac{t}{\mu^{1/\gamma}},
$$
so that
$$
r\in[1/i,R]
\quad\Longleftrightarrow\quad
t\in I_{\mu,i}
:=
\left[\frac{\mu^{1/\gamma}}{i},\mu^{1/\gamma}R\right].
$$
We further define
$$
(x(t),y(t))
:=
\left(
\mu^{1/\gamma}u\left(\frac{t}{\mu^{1/\gamma}}\right),
v\left(\frac{t}{\mu^{1/\gamma}}\right)
\right).
$$
Then, with this transformation, the pair $(x(t),y(t))$ satisfies
$$
\begin{cases}
x'
=
\varphi^{-1}\left(
\left(\dfrac{\mu^{1/\gamma}}{t}\right)^{N-1}y
\right)
=:X_{\mu}(t,y),\\[2mm]
y'
=
-\left(\dfrac{t}{\mu^{1/\gamma}}\right)^{N-1}
\left[
\dfrac{\lambda}{\mu^{2/\gamma}}x
+
\mu^{(\gamma-1)/\gamma}
g\left(\dfrac{x}{\mu^{1/\gamma}}\right)
\right]
=:-Y_{\mu}(t,x).
\end{cases}
$$
with $t \in I_{\mu,i}$. Since
$$
\frac1i\leq \frac{t}{\mu^{1/\gamma}}\leq R,
$$
the sign and monotonicity of $\varphi^{-1}$ imply, for every $y\in\mathbb{R}$,
$$
\varphi^{-1}\left(\frac{y}{R^{N-1}}\right)y
\leq
X_\mu(t,y)y
\leq
\varphi^{-1}\left(i^{N-1}y\right)y,
$$
and, for every $y \not = 0$,
$$X_\mu(t,y)y > 0.$$
Moreover, for $\mu\geq1$ and $x\in(-\delta,\delta)$, assumption $(g_2)$ gives
$$
\frac{C_1}{i^{N-1}}|x|^{\gamma} 
\leq
Y_\mu(t,x)x
\leq
R^{N-1}\left(\lambda x+C_2|x|^{\gamma-2}x\right)x.
$$
Hence Proposition~\ref{Prop21} applies with
$$
a_1(y)
=
\varphi^{-1}\left(\frac{y}{R^{N-1}}\right),
\qquad
b_1(y)
=
\varphi^{-1}\left(i^{N-1}y\right),
$$
and
$$
b_2(x)
=
\frac{C_1}{i^{N-1}}|x|^{\gamma-2}x,
\qquad
a_2(x)
=
R^{N-1}\left(\lambda x+C_2|x|^{\gamma-2}x\right).
$$
Applying Proposition~\ref{Prop21} with $k=i+1$, we obtain a number
$\hat\rho_{i+1}>0$ such that, for every
$$
\rho_{i+1}^*\in(0,\hat\rho_{i+1}],
$$
there exists $\tau_{i+1}^*>0$ such that
$$x^2 \left ( \mu^{1/\gamma} \frac{1}{i} \right) + y^2  \left ( \mu^{1/\gamma} \frac{1}{i} \right) = (\rho_{i+1}^{*})^2,$$
implies
$$\vartheta( \mu^{1/\gamma} R) - \vartheta\left (\mu^{1/\gamma} \frac{1}{i}\right) >(i+1) \pi,$$
provided $\mu \geq 1$ and $|I_{\mu,i}| = \mu^{1/\gamma} (R- 1/i) > \tau_{i+1}^*.$ Since we are free to choose $\rho_{i+1}^* \in (0,\hat \rho_{i+1}]$, we take in particular
$$
0<\Gamma_i=\rho_{i+1}^*
<
\min\left\{
\widehat\rho_{i+1},\delta,\frac1i
\right\}.
$$
In particular,
$$
\Gamma_i\to0
\qquad\text{as }i\to+\infty.
$$
By the above discussion, one has that

$$
x\left(\frac{\mu^{1/\gamma}}{i}\right)^2
+
y\left(\frac{\mu^{1/\gamma}}{i}\right)^2
=
\Gamma_i^2
$$
implies
$$
\vartheta(\mu^{1/\gamma}R)
-
\vartheta\left(\frac{\mu^{1/\gamma}}{i}\right)
>
(i+1)\pi,
$$
provided that
$ |I_{\mu,i}| = \mu^{1/\gamma}\left(R-\frac1i\right) > \tau_{i+1}^*$. Thus it is enough to take
$$
\mu> \max\left \{1,
\left(
\frac{\tau_{i+1}^*}{R-1/i}
\right)^\gamma \right\}.
$$
Since
$$
x\left(\frac{\mu^{1/\gamma}}{i}\right)^2
+
y\left(\frac{\mu^{1/\gamma}}{i}\right)^2
=
\mu^{2/\gamma}u(1/i)^2+v(1/i)^2,
$$
and the standard angular coordinate $\vartheta$ of $(x(t),y(t))$ coincides with the scaled angular coordinate $\widetilde\theta$ defined in the statement, we obtain
$$
\widetilde\theta(R)-\widetilde\theta(1/i)>(i+1)\pi,
$$
yielding the conclusion. Then, for each $i$, let $\bar\mu_i$ be sufficiently large so that the previous conclusion holds for every $\mu\geq\bar\mu_i$. Finally, replacing these thresholds, if necessary, by a strictly increasing sequence satisfying
$$
\widetilde\mu_i\geq\overline\mu_i,
\qquad
\widetilde\mu_i\to+\infty,
$$
does not affect the conclusion. This finishes the proof.
\end{proof}

Let $\lambda<\lambda_\ell^{\mathrm{rad}}$ and fix $i\geq\ell$
sufficiently large. For every
\[
\mu\in[\widetilde\mu_i,\widetilde\mu_{i+1}),
\]
we first construct a shooting parameter
$\hat\alpha_i=\hat\alpha_i(\mu,\lambda)$ satisfying the threshold
condition \eqref{eq:threshold} in Lemma~\ref{claimellipse}. To this end, consider the map
\[
F_{\mu,i}(\alpha)
:=
\mu^{2/\gamma}u_\alpha(1/i)^2+v_\alpha(1/i)^2.
\]
By continuous dependence on $\alpha$ (see Lemma \ref{continuityalpha}), $F_{\mu,i}$ is
continuous and
\[
F_{\mu,i}(0)=0.
\]
On the other hand, taking $\alpha=R+1$ and recalling that
$|u_\alpha'|<1$, we have
\[
u_{R+1}(r)
\geq
R+1-r
\geq1,
\qquad r\in[0,R],
\]
and hence
\[
F_{\mu,i}(R+1)
\geq
\mu^{2/\gamma}.
\]
Since $\Gamma_i<1/i \leq 1$ and $\mu>1$, we therefore have
\[
F_{\mu,i}(0)<\Gamma_i^2<F_{\mu,i}(R+1).
\]
By continuity, there exists
\[
\hat\alpha_i(\mu,\lambda)\in(0,R+1)
\]
such that
\begin{equation}
\label{eq:threshold-alpha}
\mu^{2/\gamma}
u_{\hat\alpha_i}(1/i)^2
+
v_{\hat\alpha_i}(1/i)^2
=
\Gamma_i^2.
\end{equation}
Set
\[
\bar u_\mu:=u_{\hat\alpha_i(\mu,\lambda)},
\qquad
\bar v_\mu:=v_{\hat\alpha_i(\mu,\lambda)}.
\]
By Lemma~\ref{claimellipse},
\[
\widetilde\theta_{\hat\alpha_i}(R)
-
\widetilde\theta_{\hat\alpha_i}(1/i)
>
(i+1)\pi.
\]
Since $\widetilde\theta_{\hat\alpha_i}(0)=0$ and
$\widetilde\theta_{\hat\alpha_i}$ is strictly increasing, we have
$\widetilde\theta_{\hat\alpha_i}(1/i)>0$, and therefore
\[
\widetilde\theta_{\hat\alpha_i}(R)>(i+1)\pi.
\]
As positive rescalings of the axes preserve comparisons with integer multiples of $\pi/2$ (see the discussion in Section \ref{subsec:polarcoordinates}), the same comparison holds for the standard angular coordinate:
\[
\theta_{\hat\alpha_i}(R)>(i+1)\pi.
\]
Since $i\geq\ell$, this yields
\begin{equation}
\label{eq:large-rotation-threshold}
\theta_{\hat\alpha_i}(R)
>
\frac{\pi}{2}+(\ell-1)\pi.
\end{equation}
On the other hand, by \eqref{eq:small-rotation}, there exists
$\alpha_*(\lambda,\mu)>0$ such that
\[
\theta_\alpha(R)
<
\frac{\pi}{2}+(\ell-1)\pi
\qquad
\text{for every }\alpha\in(0,\alpha_*).
\]
Since the map $\alpha\mapsto\theta_\alpha(R)$
is continuous, \eqref{eq:large-rotation-threshold} and the intermediate
value theorem yield a parameter
\[
\alpha_{\mu,s}^\ell
\in
(0,\hat\alpha_i(\mu,\lambda))
\]
such that
\[
\theta_{\alpha_{\mu,s}^\ell}(R)
=
\frac{\pi}{2}+(\ell-1)\pi.
\]
The corresponding solution of \eqref{system} with $u_\mu(0) = \alpha_{\mu,s}^{\ell}$, satisfies $u_\mu(R) = 0$ and has exactly $\ell-1$ zeros in $(0,R)$. This, consequently, implies that $u_\mu$ is also a solution of \eqref{eq:radial}. Moreover,
\begin{equation}
\label{eq:shooting-order}
0<u_\mu(0)
=
\alpha_{\mu,s}^\ell
<
\hat\alpha_i(\mu,\lambda)
=
\bar u_\mu(0).
\end{equation}
For every sufficiently large $\mu$, let $i(\mu)$ be the unique integer
such that
\[
\widetilde\mu_{i(\mu)}
\leq
\mu
<
\widetilde\mu_{i(\mu)+1}.
\]
Since $\widetilde\mu_i\to+\infty$, we have
\[
i(\mu)\to+\infty
\qquad\text{as }\mu\to+\infty.
\]
From \eqref{eq:threshold-alpha},
\[
|\bar u_\mu(1/i(\mu))|
\leq
\frac{\Gamma_{i(\mu)}}{\mu^{1/\gamma}}<\frac{1}{\mu^{1/\gamma}},
\]
and therefore, since $\bar u_\mu$ is $1$-Lipschitz,
\[
|\bar u_\mu(0)|
\leq
|\bar u_\mu(1/i(\mu))|
+
\frac{1}{i(\mu)}
\longrightarrow0\quad\text{as }\mu\to +\infty.
\]
By \eqref{eq:shooting-order}, the previous estimate immediately gives
\[
u_\mu(0)\longrightarrow0
\qquad\text{as }\mu\to+\infty.
\]
To conclude, let $\mu_n\to+\infty$ be arbitrary. Since the functions
$u_{\mu_n}$ are uniformly bounded and $1$-Lipschitz, the Arzelà--Ascoli theorem yields a subsequence, still denoted by $u_{\mu_n}$ for simplicity, such that
\[
u_{\mu_n}\longrightarrow u_\infty
\qquad\text{uniformly on }[0,R].
\]
The convergence of the initial values gives $u_\infty(0)=0$. If
$u_\infty\not\equiv0$, then, as proved in the first part of the proof of Theorem~\ref{thm:main-dichotomy}, the corresponding limiting pair
$(u_\infty,w_\infty)$ satisfies
\[
u_\infty(r)^2+w_\infty(r)^2\neq0
\qquad\text{for every }r\in[0,R].
\]
Since $w_\infty(0)=0$, this would imply $u_\infty(0)\neq0$, a contradiction. Therefore every cluster point of the family $\{u_\mu\}_\mu$ is the zero function. It follows that
\[
u_\mu\longrightarrow0
\qquad\text{uniformly on }[0,R]
\]
as $\mu\to+\infty$. Hence the family $\{u_\mu\}_\mu$ realizes alternative~(1) of Theorem~\ref{thm:main-dichotomy}. This proves the second part of Theorem~\ref{thm:realization}, with
\[
u^{(s)}_{\mu,\ell}:=u_\mu.
\]

\subsection{Construction of the large solutions}
\label{largesols}
To prove the existence of a family of solutions satisfying alternative (2) of Theorem \ref{thm:main-dichotomy}, we shall apply a \emph{backward shooting} method. That is, we shall consider the following problem:

\begin{equation}
\label{backwardshooting}
\left\{
\begin{aligned}
u' &= \varphi^{-1}\!\left(\frac{v}{r^{N-1}}\right),
&& r \in (0,R),\\
v' &= -r^{N-1}\bigl(\lambda u+\mu g(u)\bigr),
&& r \in (0,R),\\
u(R) &= 0,\\
v(R) &= \eta.
\end{aligned}
\right.
\end{equation}
Here $\eta>0$ when $\ell$ is even, and $\eta<0$ when $\ell$ is odd.
To fix the ideas, we will consider the case in which $\ell$ is even, the
other one being analogous.

By the standard Cauchy-Lipschitz theory for ODEs, the Cauchy problem \eqref{backwardshooting} is well-posed and, for every $\eta \in \R$, has a unique solution $(u_\eta,v_\eta)$. Moreover, the solutions $(u_\eta,v_\eta)$ of
\eqref{backwardshooting} give rise to solutions of \eqref{eq:radial} if and
only if $v_\eta(0)=0$. Furthermore, the same continuous-dependence result on the value $\eta$ holds, as
in Lemma~\ref{continuityalpha}. That is, if $\eta \to \bar \eta$, then
$$(u_\eta,v_\eta) \to (u_{\bar\eta}, v_{\bar\eta}) \quad \text{uniformly on } [0,R]. $$
In particular, the solutions of \eqref{backwardshooting} depend
continuously on $\eta$. As for \eqref{system}, if $\eta=0$, then
$(u_\eta,v_\eta)\equiv(0,0)$.

For $\eta\neq0$, we introduce the clockwise polar coordinates
$$
u_\eta(r)=\varrho_\eta(r)\cos\vartheta_\eta(r),
\qquad
v_\eta(r)=-\varrho_\eta(r)\sin\vartheta_\eta(r),
$$
with $\vartheta_\eta(R)=-\pi/2$ when $\eta>0$ and $\vartheta_\eta(R) = \pi/2$ when $\eta <0$, respectively. A direct computation gives
\[
\vartheta_\eta'(r)
=
\frac{\sin^2\vartheta_\eta(r)}
{r^{N-1}\sqrt{1+(v_\eta(r)/r^{N-1})^2}}
+
r^{N-1}
\left(
\lambda\cos^2\vartheta_\eta(r)
+
\frac{\mu g(u_\eta(r))u_\eta(r)}
{\varrho_\eta(r)^2}
\right).
\]
In particular, by $(g_3)$, $\vartheta_\eta$ is strictly increasing.
For $\mu>0$, we also denote by
$\widetilde\varrho_\eta,\widetilde\vartheta_\eta$
the scaled coordinates given by \eqref{scaledpolar},
with $(u_\alpha,v_\alpha)$ replaced by $(u_\eta,v_\eta)$
and with $\beta_1=\mu^{-1/\gamma}$, $\beta_2=1$.
We normalize the scaled angle by
\[
    \widetilde\vartheta_\eta(R)=\vartheta_\eta(R).
\]
Again, $\widetilde\vartheta_\eta$ is strictly increasing,
and both $\vartheta_\eta$ and
$\widetilde\vartheta_\eta$ have the same comparisons
with integer multiples of $\pi/2$.

We can now prove the following result.

\begin{lemma}
\label{etatilde}
There exists $\tilde\eta=\tilde\eta(\lambda,\mu)>0$ such that, for every
$|\eta|>\tilde\eta$,
\[
0\leq
\vartheta_\eta(R)-\vartheta_\eta(0)
<
\frac{\pi}{2}.
\]
\end{lemma}

\begin{proof}
Indeed, due to the bound $|u_\eta(r)|\leq R$ for every $r\in[0,R]$, and setting
\[
M_R:=\max_{|s|\leq R}|g(s)|,
\]
for
\[
|\eta|>
\frac{R^N}{N}\bigl(\lambda R+\mu M_R\bigr)=:\tilde \eta
\]
we have, for every $r\in[0,R]$,
\[
|\eta|
>
\left|
\int_r^R
s^{N-1}
\bigl(\lambda u_\eta(s)+\mu g(u_\eta(s))\bigr)\,ds
\right|.
\]
Consequently, for every $ r \in [0,R]$,
\begin{equation}
\label{positivevbackward}
v_\eta(r)
= 
\eta+
\int_r^R
s^{N-1}
\bigl(\lambda u_\eta(s)+\mu g(u_\eta(s))\bigr)\,ds
\begin{cases}
    >0 \quad \text{if } \eta >\tilde \eta \\
    < 0 \quad \text{if } \eta < - \tilde \eta.
\end{cases}
\end{equation}
It follows that, when $\eta > \tilde \eta,$ $u_\eta'(r)>0$ for $r\in(0,R)$. Since $u_\eta(R)=0$,
we have $u_\eta(r)<0$ for every $r\in[0,R)$. Therefore, the trajectory
$(u_\eta,v_\eta)$ remains confined in the second quadrant of the phase
plane. Similarly, when $\eta < - \tilde \eta $, $u_\eta'(r) <0$ for every $r \in (0,R)$. Reasoning exactly as before, this means that the trajectory $(u_\eta, v_\eta)$ remains confined in the fourth quadrant of the phase plane. These two facts, together with the monotonicity of $\vartheta_\eta$, yield
\[
0\leq
\vartheta_\eta(R)-\vartheta_\eta(0)
<
\frac{\pi}{2},
\]
which concludes the proof.
\end{proof}

We now prove a variant of Lemma~\ref{claimellipse}, considered for this backward shooting approach. Namely:

\begin{lemma}
\label{etastar}
Let $\lambda\geq0$. For every integer $i\geq1$ and every
$r_0\in(0,R)$, there exist
\[
\hat\mu_i=\hat\mu_i(\lambda,r_0)>0,
\qquad
\Gamma_i=\Gamma_i(\lambda,r_0)>0,
\]
such that, for every $\mu>\hat\mu_i$, every solution $(u(r),v(r))$ of the differential system in \eqref{backwardshooting}, defined on $[r_0,R]$ and satisfying
\[
\mu^{2/\gamma}u(R)^2+v(R)^2=\Gamma_i^2,
\]
fulfills
\begin{equation}
\label{vartheta}
\widetilde\vartheta(R)-\widetilde\vartheta(r_0)
>
(i+1)\pi,
\end{equation}
where $\gamma$ is as in $(g_2)$ and the scaled angular coordinate
$\widetilde\vartheta$ is defined by
\[
(u(r),v(r))
=
\left(
\mu^{-1/\gamma}\rho(r)\cos\widetilde\vartheta(r),
-\rho(r)\sin\widetilde\vartheta(r)
\right).
\]

\end{lemma}

\begin{proof}
The proof follows the same lines of the proof of
Lemma~\ref{claimellipse}, this time using
Proposition~\ref{timeinverse} instead of Proposition~\ref{Prop21}, and
$r_0$ instead of $1/i$. By assumption $(g_3)$, the global sign condition required in Proposition~\ref{timeinverse} is satisfied. As the proof is valid for any $r_0\in(0,R)$, the constants
$\hat\mu_i$ and $\Gamma_i$ depend also on the choice of $r_0$.
\end{proof}

Fix $r_0\in(0,R)$, and suppose, for the moment, that $\ell$ is even; take $i=\ell$ in the previous lemma. By the same lemma, for every sufficiently large $\mu$, choosing the backward shooting parameter
\[
\eta^-:=\Gamma_\ell>0,
\]
where $\Gamma_\ell=\Gamma_\ell(\lambda,r_0)$,
the corresponding trajectory satisfies the terminal condition of Lemma~\ref{etastar}. Hence
\[
\widetilde\vartheta_{\eta^-}(R)-\widetilde\vartheta_{\eta^-}(r_0)>(\ell+1)\pi.
\]
Since $\widetilde\vartheta_{\eta^-}$ is strictly increasing, we obtain the stronger estimate
\[
\widetilde\vartheta_{\eta^-}(R)-\widetilde\vartheta_{\eta^-}(0)>(\ell+1)\pi.
\]
Moreover, because $\eta^->0$, we may normalize
\[
\widetilde\vartheta_{\eta^-}(R)=\vartheta_{\eta^-}(R)=-\frac{\pi}{2}.
\]
This implies that 
$$\widetilde \vartheta_{\eta^-}(0) < -(\ell+1) \pi - \frac \pi 2,  $$
which, by the invariance of comparisons with integer multiples of $\pi/2$ under positive rescalings of the axes, yields
$$\vartheta_{\eta^-}(0) < -(\ell+1) \pi - \frac \pi 2,$$
and, consequently,
\[
\vartheta_{\eta^-}(R)-\vartheta_{\eta^-}(0)>(\ell+1)\pi
>
\frac{\pi}{2}+(\ell-1)\pi.
\]
On the other hand, by Lemma~\ref{etatilde}, we can choose
\[
\eta^+>\max\{\tilde\eta(\lambda,\mu),\Gamma_\ell\}
\]
such that
\[
0\leq
\vartheta_{\eta^+}(R)-\vartheta_{\eta^+}(0)
<
\frac{\pi}{2}
<
\frac{\pi}{2}+(\ell-1)\pi.
\]
Exploiting the continuity of the map
$
\eta\longmapsto
\vartheta_\eta(R)-\vartheta_\eta(0)
$ and the intermediate value theorem, there exists
\[
\eta_{\mu,\ell}\in(\Gamma_\ell,\eta^+)
\]
such that
\[
\vartheta_{\eta_{\mu,\ell}}(R)
-
\vartheta_{\eta_{\mu,\ell}}(0)
=
\frac{\pi}{2}+(\ell-1)\pi.
\]
As before, since $\ell$ is even, we normalize
$ \vartheta_{\eta_{\mu,\ell}}(R)=-\frac{\pi}{2}$, and therefore
\[
\vartheta_{\eta_{\mu,\ell}}(0)
=
-\ell\pi.
\]
It follows that
\[
v_{\eta_{\mu,\ell}}(0)=0,
\qquad
u_{\eta_{\mu,\ell}}(0)>0.
\]
Moreover, since $\vartheta_{\eta_{\mu,\ell}}$ is strictly increasing, the corresponding solution has exactly $\ell-1$ zeros in $(0,R)$. Hence, $u_{\eta_{\mu,\ell}}$ is a solution of \eqref{eq:radial} with the required nodal properties.

If $\ell$ is odd, the same construction is carried out for $\eta<0$.
For every $\eta<0$, we normalize
\[
\widetilde\vartheta_\eta(R)=\vartheta_\eta(R)=\frac{\pi}{2}.
\]
Taking $\eta^+:=-\Gamma_\ell$, Lemma~\ref{etastar} and the
monotonicity of $\widetilde\vartheta_{\eta^+}$ give
\[
\widetilde\vartheta_{\eta^+}(R)
-\widetilde\vartheta_{\eta^+}(0)>(\ell+1)\pi,
\]
and therefore
\[
\widetilde\vartheta_{\eta^+}(0)
<\frac{\pi}{2}-(\ell+1)\pi.
\]
Again, by the invariance of comparisons with integer multiples of $\pi/2$
under positive rescalings of the axes, we obtain
\[
\vartheta_{\eta^+}(0)<\frac{\pi}{2}-(\ell+1)\pi,
\]
and hence
\[
\vartheta_{\eta^+}(R)-\vartheta_{\eta^+}(0)>(\ell+1)\pi.
\]
On the other hand, for $\eta<0$ with $|\eta|$ sufficiently large,
the estimate following Lemma~\ref{etatilde} gives an angular variation
smaller than $\pi/2$. The intermediate value theorem therefore provides
a parameter $\eta_{\mu,\ell}<-\Gamma_\ell$ with
\[
\vartheta_{\eta_{\mu,\ell}}(R)-\vartheta_{\eta_{\mu,\ell}}(0)
=\frac{\pi}{2}+(\ell-1)\pi.
\]
Since $\vartheta_{\eta_{\mu,\ell}}(R)=\pi/2$, we obtain
\[
\vartheta_{\eta_{\mu,\ell}}(0)=(1-\ell)\pi,
\]
and again $v_{\eta_{\mu,\ell}}(0)=0$ and $u_{\eta_{\mu,\ell}}(0)>0$.
Notice that no oddness assumption on $g$ is used here.

\medskip
In what follows, we denote the solution obtained in either case by $u_\mu$ and, correspondingly,
\[
v_\mu=r^{N-1}\varphi(u_\mu').
\]
By construction,
\begin{equation}
\label{etamu}
|v_\mu(R)|=|\eta_{\mu,\ell}|>\Gamma_\ell>0,
\end{equation}
with $\Gamma_\ell$ independent of $\mu$.
We are now ready to conclude. Indeed, since $\{u_\mu\}_\mu$ is a family of solutions of \eqref{eq:radial}, the functions $u_\mu$ are $1$-Lipschitz. Therefore, by means of the Arzelà--Ascoli theorem, up to a subsequence,
\[
u_\mu\rightarrow u_\infty
\qquad\text{uniformly on }[0,R]
\]
as $\mu\to+\infty$. Moreover, from \eqref{etamu} and the monotonicity and oddness of $\varphi^{-1}$, we have
\[
|u_\mu'(R)|
=
\left|\varphi^{-1}
\left(
\frac{v_\mu(R)}{R^{N-1}}
\right)\right|
\geq
\varphi^{-1}
\left(
\frac{\Gamma_\ell}{R^{N-1}}
\right)
>0,
\]
which, in particular, yields
\begin{equation}
\label{uprimeR}
\liminf_{\mu\to+\infty}|u_\mu'(R)|>0.
\end{equation}
Now, assume by contradiction that $u_\infty\equiv0$. Then, by Theorem~\ref{thm:main-dichotomy} applied to $\{u_\mu\}_{\mu\geq M}$, for some $\bar r\in(0,R)$ we have
\[
u_\mu'\rightarrow0
\qquad\text{uniformly in }[\bar r,R].
\]
In particular,
\[
|u_\mu'(R)|\rightarrow0,
\]
contradicting \eqref{uprimeR}. Therefore, we must be in case~(2) of Theorem \ref{thm:main-dichotomy}, which concludes the proof.

\section{Numerical simulations}
\label{sec:num}

We conclude this work by presenting several numerical simulations,
obtained through a shooting method implemented in \texttt{MATLAB}.
The radial initial-value problem was integrated using the solver
\texttt{ode45}, starting at a small positive radius to avoid the
singularity at the origin, and the shooting parameter was adjusted
by bisection to satisfy the boundary condition $u(R)=0$.
Solutions were then selected according to their number of interior zeros.

These simulations illustrate the asymptotic behaviour established
in the previous sections. Throughout this section, solid curves
represent large solutions, while dashed curves represent small solutions.
We recall the notation introduced in \cite[Section~3.1.3]{BoCoZi26}:
for $\lambda<\lambda_j^{\mathrm{rad}}$, $\mu_j^\star(\lambda)$ denotes
the value of $\mu$ at the turning point of the $j$-th bifurcation branch
with respect to $\mu$, where the small and large branches of solutions
with $j-1$ interior zeros meet
(see also \cite[Figure~2]{BoCoZi26}).
Thus, for $\mu>\mu_j^\star(\lambda)$, the corresponding small and large
solutions coexist. On the other hand, if $\lambda>\lambda_j^{\mathrm{rad}}$,
then, by definition, $\mu_j^\star(\lambda)=0$, and no branch of small
solutions exists, while the branch of large solutions persists
for all $\mu\geq0$.

\begin{figure}[ht]
    \centering

    \begin{subfigure}{0.48\textwidth}
        \centering
        \includegraphics[width=\linewidth]{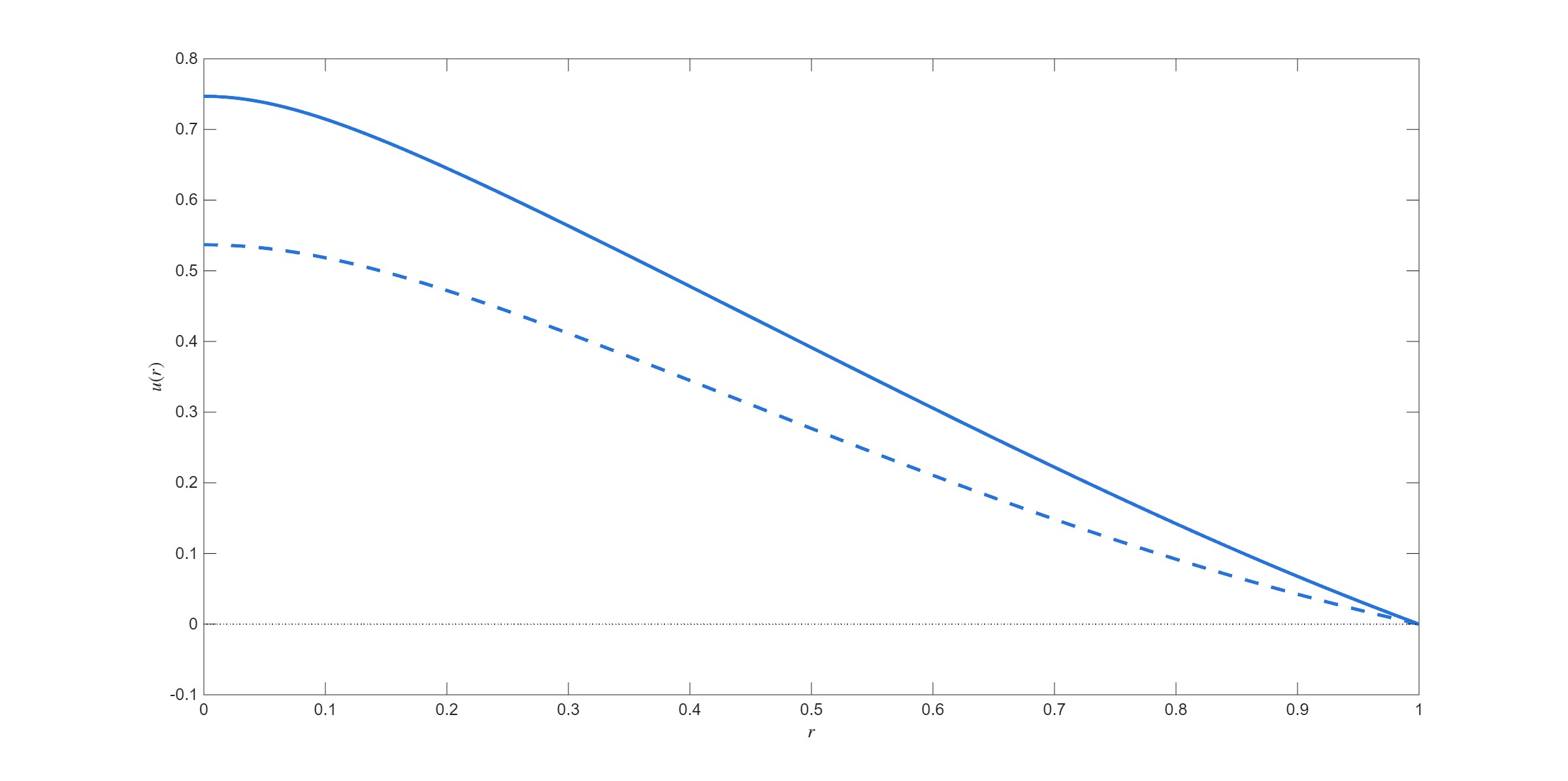}
        \caption{$\mu \approx 38$}
        \label{fig:sub1}
    \end{subfigure}
    \hfill
    \begin{subfigure}{0.48\textwidth}
        \centering
        \includegraphics[width=\linewidth]{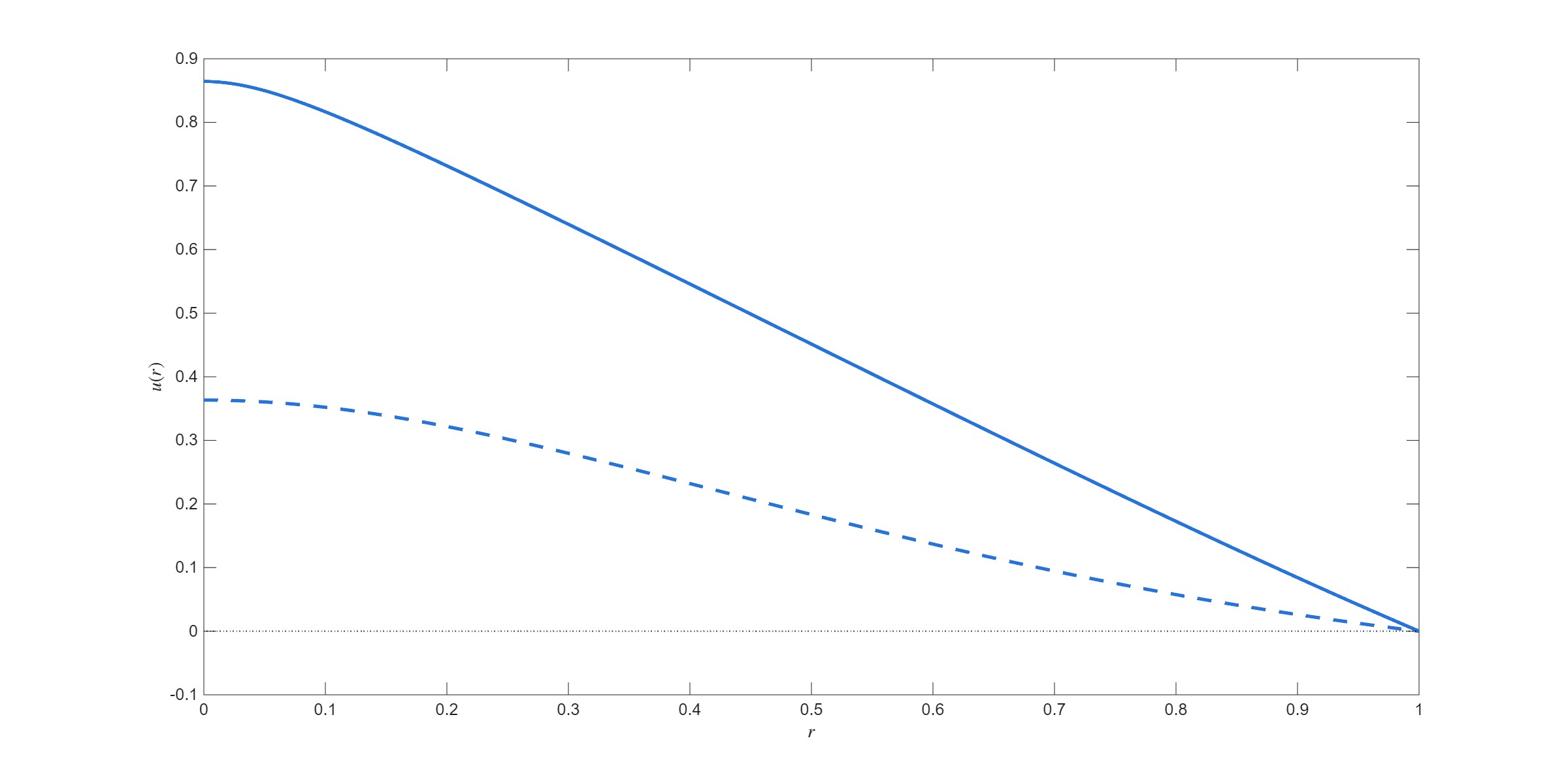}
        \caption{$\mu \approx 50$}
        \label{fig:sub2}
    \end{subfigure}
    \hfill
    \begin{subfigure}{0.48\textwidth}
        \centering
        \includegraphics[width=\linewidth]{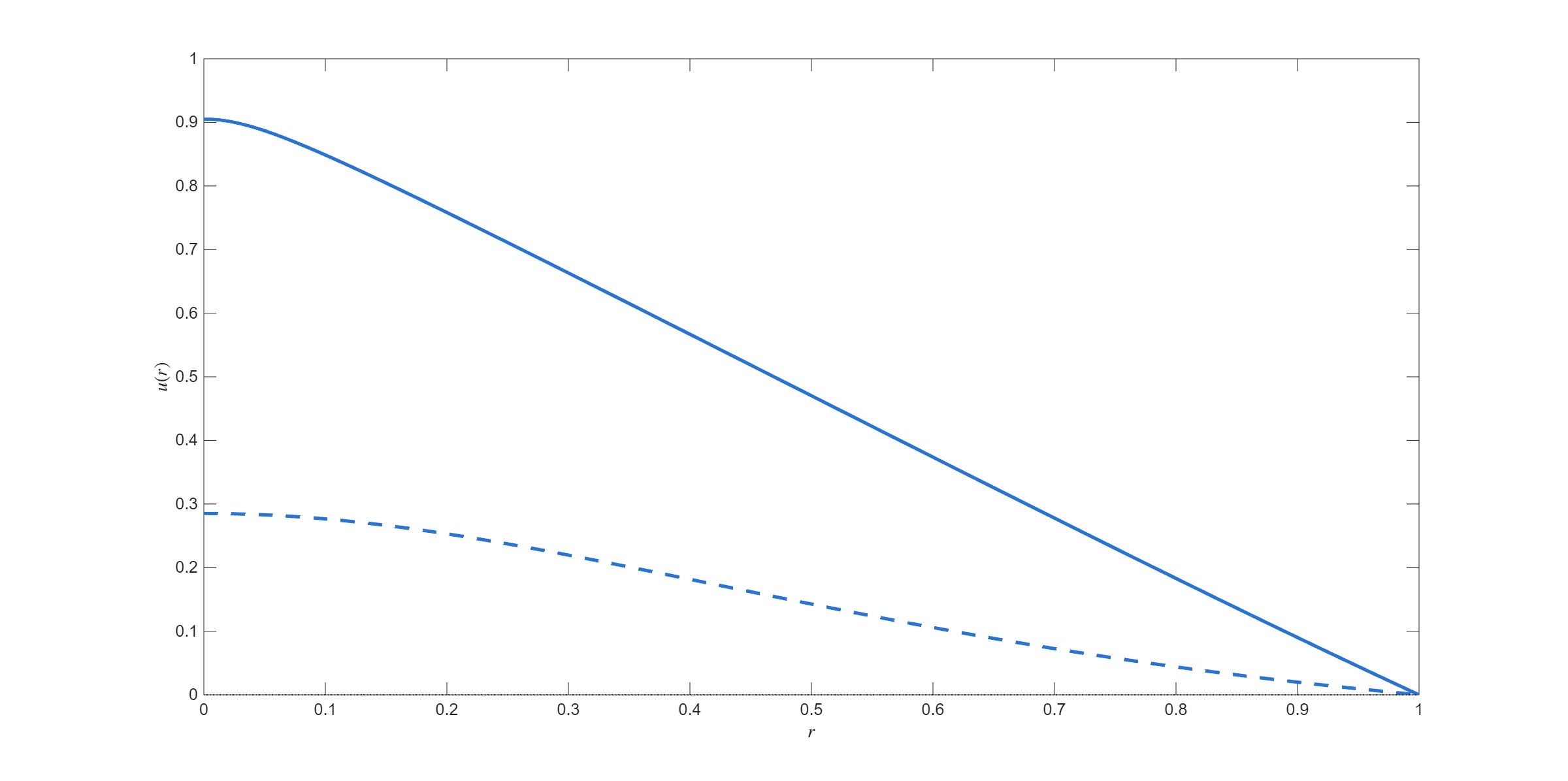}
        \caption{$\mu \approx62$}
        \label{fig:sub3}
    \end{subfigure}

    \caption{Profiles of the large and small $0$-node solutions for $\mu \approx \mu_1^\star$, $\mu \approx 1.3\mu_1^\star$, and $\mu \approx 1.6\mu_1^\star$. The parameters are $N=3$, $\lambda=1$, $R=1$, and $g(u)=|u| u$.}
    \label{0nodes}
\end{figure}
\begin{figure}[ht]
    \centering

    \begin{subfigure}{0.48\textwidth}
        \centering
        \includegraphics[width=\linewidth]{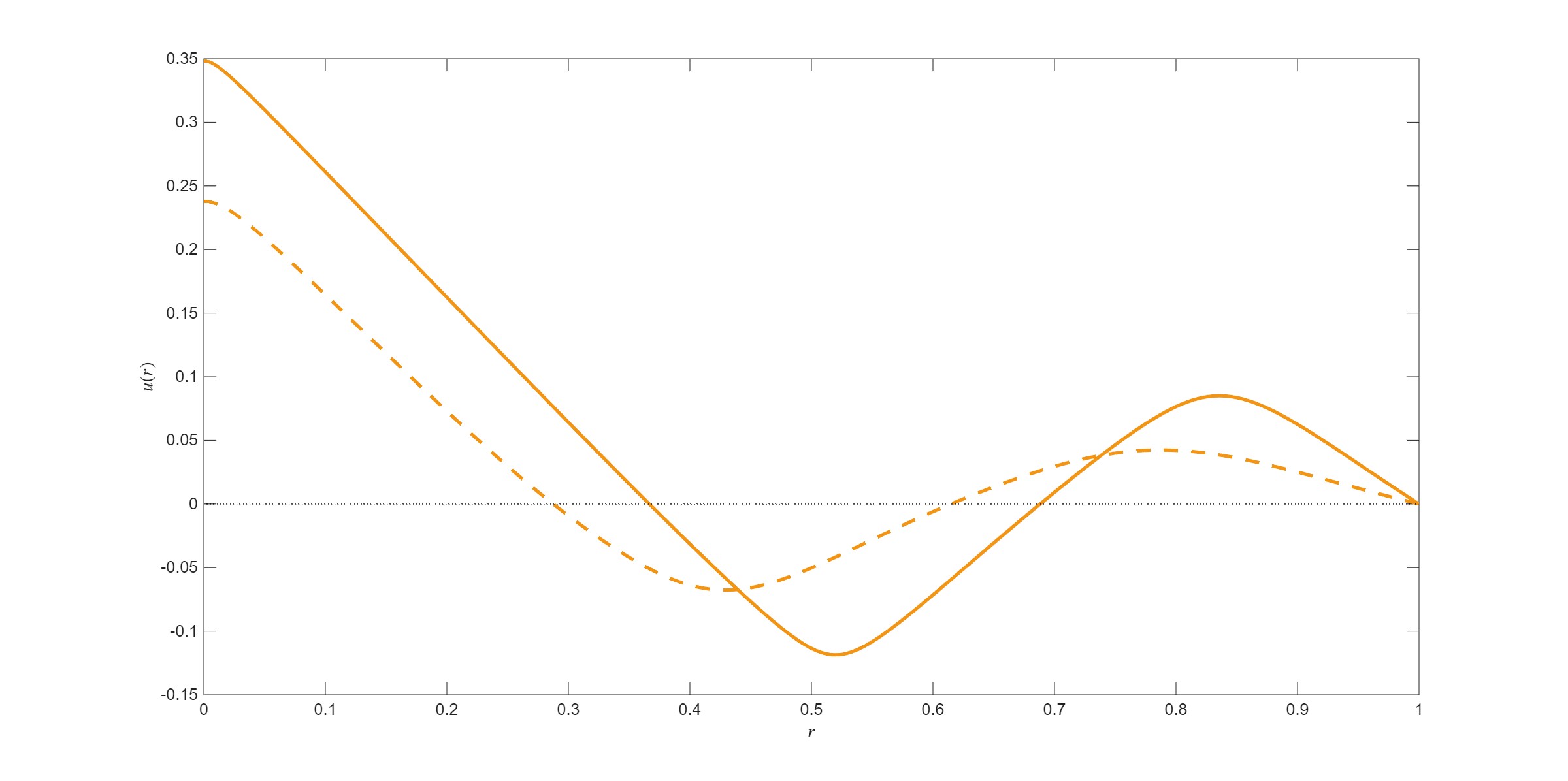}
        \caption{$\mu \approx 1947$}
        \label{fig:sub1}
    \end{subfigure}
    \hfill
    \begin{subfigure}{0.48\textwidth}
        \centering
        \includegraphics[width=\linewidth]{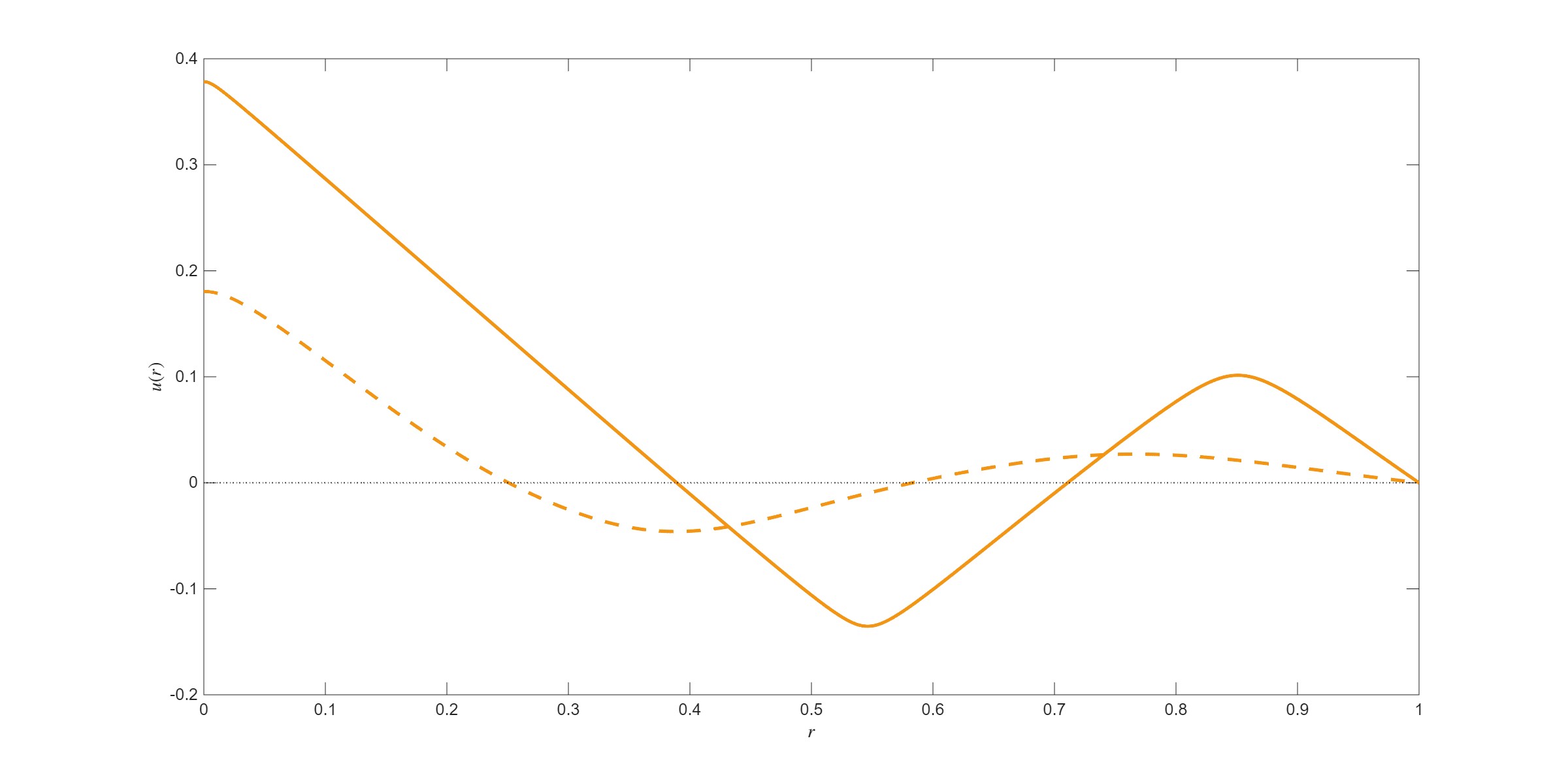}
        \caption{$\mu \approx 2532$}
        \label{fig:sub2}
    \end{subfigure}
    \hfill
    \begin{subfigure}{0.48\textwidth}
        \centering
        \includegraphics[width=\linewidth]{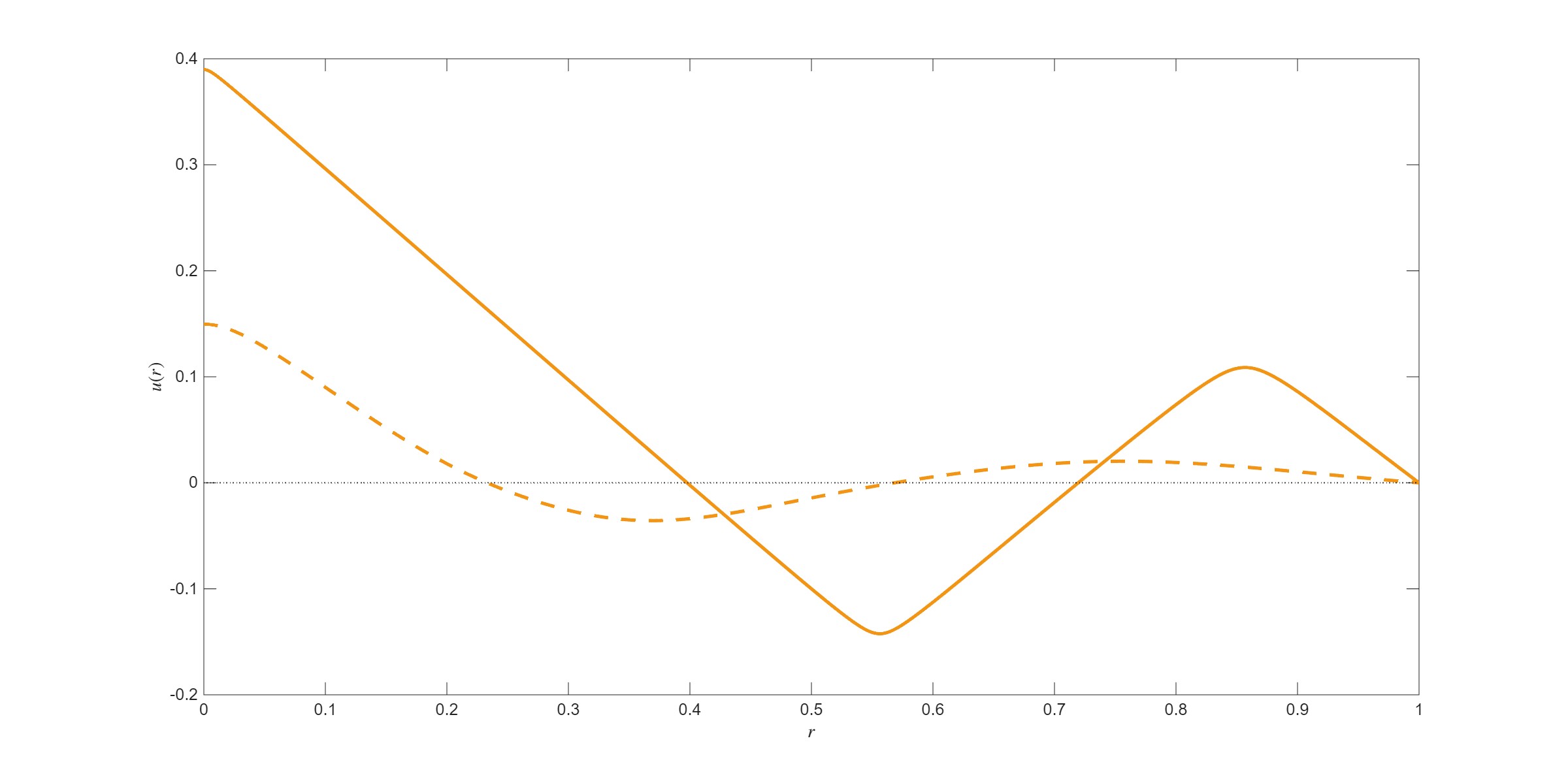}
        \caption{$\mu \approx3116$}
        \label{fig:sub3}
    \end{subfigure}

    \caption{Profiles of the large and small $2$-node solutions for $\mu \approx \mu_3^\star$, $\mu \approx 1.3\mu_3^\star$, and $\mu \approx 1.6\mu_3^\star$. The parameters are $N=3$, $\lambda=1$, $R=1$, and $g(u)=|u| u$.}
    \label{2nodes}
\end{figure}

We begin with the large and small $0$-node solutions. Their profiles for $\mu \approx \mu_1^\star$, $\mu \approx 1.3\mu_1^\star$, and $\mu \approx 1.6\mu_1^\star$ are displayed in Figure~\ref{0nodes}. It is readily observed that, as $\mu$ increases, the initial value of the large solution increases, going from $\approx 0.75$ to $\approx 0.9$, whereas that of the small solution decreases, going from $\approx0.54$ to $\approx0.29$. Similarly, in Figure~\ref{2nodes} we show the large and small 2-node solutions, for $\mu \approx \mu_3^\star$, $\mu \approx 1.3 \mu_3^\star$ and $\mu \approx 1.6 \mu_3^\star$. As in the previous case, the initial value of the large solution goes up as $\mu$ increases, going from $\approx 0.35$ to $\approx 0.4$, and the initial value of the small solution decreases, going from $\approx 0.24$ to $\approx 0.15$.

As $\mu$ increases, the large solutions exhibit increasingly piecewise affine profiles, as in alternative (2) of Theorem~\ref{thm:main-dichotomy}. In particular, Figure~\ref{allsols} shows the $0$-, $1$-, and $2$-node large solutions for large values of $\mu$, displaying the behaviour corresponding to alternative (2) of Theorem \ref{thm:main-dichotomy}.
\begin{figure}[h]
    \centering
    \includegraphics[width=0.7\linewidth]{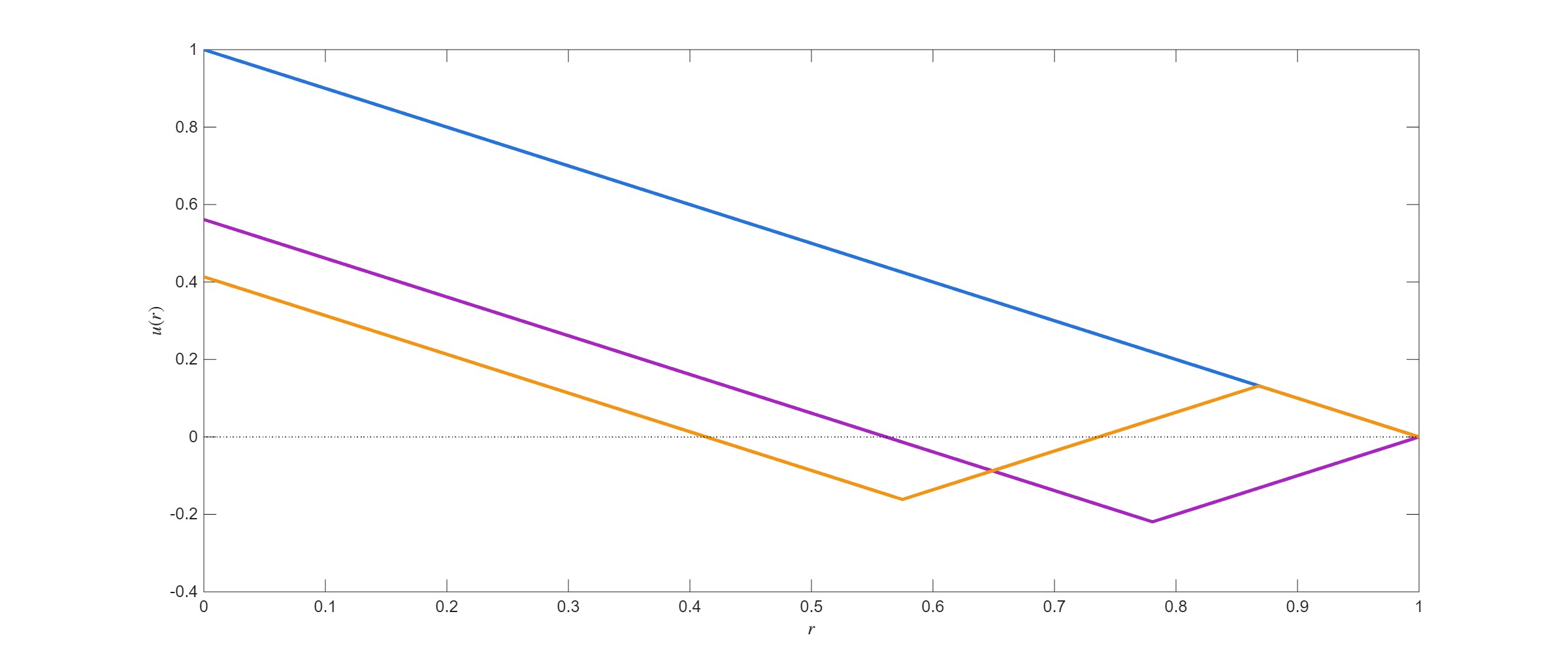}
    \caption{0-node, 1-node and 2-node large solutions for $\mu \gg1$. The parameters are $N=3$, $\lambda=1$, $R=1$, and $g(u) =|u| u$.}
    \label{allsols}
\end{figure}
\appendix
\section{Technical results}
\label{appendix1}
In this section, we are going to present an auxiliary result dealing with a planar Hamiltonian system of the type
\begin{equation}
\label{hamiltonian}
    \begin{cases}
        x' = X(t,y) \\
        y' = -Y(t,x),
    \end{cases}
\end{equation}
where $X, Y : I \times \mathbb{R} \to \mathbb{R}$ are locally Lipschitz continuous functions and $I\subseteq \mathbb R$ is an interval. Roughly speaking, we are going to prove that, under suitable (local) sign-conditions, the number of rotations of the planar path $(x(t), y(t))$ around the origin becomes arbitrarily large as the width of the interval $I$ increases.

To this end, we will write (whenever possible, namely for $x(t)^2 + y(t)^2 > 0$) solutions of \eqref{hamiltonian} in (clockwise) polar coordinates as
\begin{equation}
    x(t) = \rho(t)\cos\theta(t), \qquad y(t) = -\rho(t)\sin\theta(t), \label{eq:2.2}
\end{equation}
with $\rho(t) > 0$. Notice that, of course, the angular coordinate $\theta(t)$ is defined up to integer multiples of $2\pi$; however, the expression $\theta(t_2) - \theta(t_1)$, for any $t_1, t_2 \in I$, is uniquely determined, depending on the path $(x(t), y(t))$ only. This result is a slight variant of \cite[Lemma 3.2]{BoZa13} and \cite[Proposition 2.1]{BoGa19}, and we include its proof for the sake of completeness.

\begin{proposition}
\label{Prop21}
Let $\delta>0$ and let $a_i, b_i : [-\delta, \delta] \to \mathbb{R}$,  with $i = 1, 2$, be Lipschitz continuous functions such that
\begin{equation}
    0 < a_1(s)s \leq b_1(s)s \quad \text{and} \quad 0 < b_2(s)s \leq a_2(s)s
    \quad \text{for every } s \neq 0. \label{eq:2.3}
\end{equation}
Then, for every positive integer $k$, there exists $\hat \rho_k>0$ such that, for every $\rho^*_k \in (0,\hat\rho_k]$, there exists $\tau_k^* > 0$ such that, for every interval $I = [t_0, t_1] \subset \mathbb{R}$ with $|I| := t_1 - t_0 > \tau_k^*$ and for every locally Lipschitz continuous functions $X, Y : I \times \mathbb{R} \to \mathbb{R}$ satisfying
\begin{equation}
\label{first}
    a_1(y)y \leq X(t,y)y \leq b_1(y)y \quad \text{for every } t \in I,\ y \in (-\delta, \delta)
\end{equation}
\begin{equation} \label{third}
    0<X(t,y) y   \quad \quad \quad \, \, \, \,  \,\quad \text{for every } t\in I,\ y \not = 0
\end{equation}
and
\begin{equation}
\label{second}
    b_2(x)x \leq Y(t,x)x \leq a_2(x)x \quad \text{for every } t \in I,\ x \in (-\delta, \delta),
\end{equation}
it holds that any solution $(x(t), y(t))$ of \eqref{hamiltonian}, defined on $I$ and satisfying $x(t_0)^2 + y(t_0)^2 = (\rho_k^*)^2$, fulfills $x(t)^2 + y(t)^2 > 0$ for every $t \in I$ and
$$
    \theta(t_1) - \theta(t_0) > k\pi.
$$
\end{proposition}

\begin{proof}
    The idea is to adapt the proof of {\cite[Lemma 3.2]{BoZa13}}. The crucial point is to construct two spiralling planar curves controlling (from below and from above) the rotations of the planar path $(x(t), y(t))$. This can be done by gluing together pieces of level curves of suitable energy functions. Precisely, after having extended $a_i$ and $b_i$ to locally Lipschitz continuous functions on the whole real line, in such a way that \eqref{eq:2.3} still holds and that the primitives
    $$A_i(s) = \int_0^s a_i(\xi) d\xi, \quad B_i(s) = \int_0^s b_i(\xi) d \xi$$
    are coercive, one defines the energies
    $$\mathcal{E}_A(x,y) := A_1(y) + A_2(x), \quad \mathcal{E}_B(x,y):= B_1(y) + B_2(x).$$
Straightforward computations show that, along a solution $(x(t), y(t))$ of \eqref{hamiltonian} such that $\vert x(t) \vert, \vert y(t) \vert < \delta$, it holds
$$
    \frac{d}{dt}\mathcal{E}_A(x(t), y(t)) \geq 0, \text{ if } x(t)y(t) \geq 0, \qquad
    \frac{d}{dt}\mathcal{E}_A(x(t), y(t)) \leq 0, \text{ if } x(t)y(t) \leq 0
$$
and
$$
    \frac{d}{dt}\mathcal{E}_B(x(t), y(t)) \geq 0, \text{ if } x(t)y(t) \leq 0, \qquad
    \frac{d}{dt}\mathcal{E}_B(x(t), y(t)) \leq 0, \text{ if } x(t)y(t) \geq 0.
$$

Now, consider the auxiliary functions 
$$
    \mathcal{M}_+(x,y) = \begin{cases}
        \dfrac{b_1(y)x - b_2(x)y}{b_2(x)x + b_1(y)y}, & \text{if } xy \geq 0 \\
        \\
        \dfrac{a_1(y)x - a_2(x)y}{a_2(x)x + a_1(y)y}, & \text{if } xy \leq 0,
    \end{cases}
$$
and

$$
    \mathcal{M}_-(x,y) = \begin{cases}
        \dfrac{b_1(y)x - b_2(x)y}{b_2(x)x + b_1(y)y}, & \text{if } xy \leq 0 \\
        \\
        \dfrac{a_1(y)x - a_2(x)y}{a_2(x)x + a_1(y)y}, & \text{if } xy \geq 0.
    \end{cases}
$$

Moreover, consider the associated first-order differential equations
\begin{equation}
    \frac{dr_-}{d\theta} = r\,\mathcal{M}_-(r\cos\theta, -r\sin\theta), 
    \quad (r,\theta) \in \mathbb{R}_0^+ \times \mathbb{R}, \label{eq:3.17}
\end{equation}
\begin{equation}
    \frac{dr_+}{d\theta} = r\,\mathcal{M}_+(r\cos\theta, -r\sin\theta), 
    \quad (r,\theta) \in \mathbb{R}_0^+ \times \mathbb{R}. \label{eq:3.18}
\end{equation}
For every $(\theta_0, \rho_0)$ with $\theta_0 \in \mathbb{R}$ and $\rho_0 > 0$, we denote, respectively, by $r_-(\,\cdot\,;\theta_0,\rho_0)$ and $r_+(\,\cdot\,;\theta_0,\rho_0)$ the solutions of either \eqref{eq:3.17} or \eqref{eq:3.18} satisfying the initial condition $r(\theta_0) = \rho_0$.

The geometrical meaning of these solutions is the following: the maps
$$
    \theta \mapsto (r_\pm(\theta;\theta_0,\rho_0)\cos\theta,\, -r_\pm(\theta;\theta_0,\rho_0)\sin\theta)
$$
parameterize two spiralling curves, in the Cartesian plane, passing through
$$
    (x_0, y_0) := (\rho_0\cos\theta_0, -\rho_0\sin\theta_0)
$$
and obtained by alternating (along the motion from one quadrant to another) level curves of $\mathcal{E}_A$ and $\mathcal{E}_B$. Such spirals are precisely those we need to bound (from below and above) in the $(x,y)$-plane the solutions of \eqref{hamiltonian}.

Since the functions $a_i$ and $b_i$ are Lipschitz continuous, the
right-hand sides of \eqref{eq:3.17} and \eqref{eq:3.18} are locally Lipschitz
with respect to $r$ on $\mathbb{R}_0^+\times\mathbb{R}$. Hence the
uniqueness of the solutions $r^\pm(\theta)$ follows from the standard
Cauchy--Lipschitz theory. On the other hand, \eqref{eq:2.3} jointly with the coercivity of $A_i, B_i$ guarantee the global existence of the solutions. The $2\pi$-periodicity of $\mathcal{M}_\pm$ in the $\theta$-variable implies
$$
    r_+(\theta+2\pi;\theta_0 + 2\pi, \rho_0) = r_+(\theta;\theta_0,\rho_0), 
    \quad \text{for all }\,\theta,\theta_0 \in \mathbb{R} \text{ and } \rho_0 > 0.
$$
At this point, we define, for every positive integer $k$ and for every $\rho_0 > 0$, the parameters
$$
    m_k^*(\rho_0) := \inf_{\substack{\theta_0 \in [0,2\pi] \\ \theta \in [\theta_0,\theta_0+2\pi k]}} r_-(\theta;\theta_0,\rho_0), 
    \qquad
    M_k^*(\rho_0) := \sup_{\substack{\theta_0 \in [0,2\pi] \\ \theta \in [\theta_0,\theta_0+2\pi k]}} r_+(\theta;\theta_0,\rho_0).
$$
The number $M_k^*(\rho_0)$ provides an upper bound for the modulus of a spiral associated with \eqref{eq:3.18} and departing from the circumference $\rho = \rho_0$, while performing an angular displacement of $2k\pi $. Similarly, $m_k^*(\rho_0)$ gives a lower bound for the modulus of the spiral associated with \eqref{eq:3.17}. Note that, from the geometrical interpretation, it is clear that both $M^*_k, m_k^*$ are monotone increasing in $\rho_0$.  Now, for $\delta > 0$ as in \eqref{first}, and any fixed $k \geq 1$, we choose $\hat \rho_k \in ]0,\delta[$ such that
$$
    M_{k+1}^*(\hat \rho_k) < \delta,
$$
and accordingly,
$$M_{k+1}^*(\rho_k^*) \leq M_{k+1}^*(\hat \rho_k)<\delta$$
Subsequently, we fix two numbers $\check{r}_k$ and $\hat{r}_k$ such that
\begin{equation}
\label{rj}
    0 < \check{r}_k < m_{k+1}^*(\rho_k^*) \leq \rho _k^* \leq M_{k+1}^*(\rho_k^*) < \hat{r}_k < \delta. 
\end{equation}
Once we have chosen $\hat{r}_k$ and $\check{r}_k$, we can define 
\begin{equation}
\label{deltaj}
    \delta_k^* := \inf_{{\check{r}_k \leq \sqrt{x^2+y^2} \leq \hat{r}_k}} \frac{b_2(x)x + a_1(y)y}{x^2 + y^2} > 0. 
\end{equation}
At last, we set
\begin{equation}
    \tau_k^* := \frac{2\pi k}{\delta_k^*} \label{eq:3.21}
\end{equation}
and, finally, we are in the position to verify that the conclusion of the proposition holds true.

Without loss of generality, we suppose that $I = [0, \tau]$ with
$$
    \tau = |I| > \tau_k^*.
$$
Let also $X,Y$ satisfying \eqref{first} and \eqref{second} respectively, and let $ (x(t), y(t))$ be a solution of \eqref{hamiltonian} (which we suppose to be defined on $I$) with
$$
x(0)^2 + y(0)^2= (\rho^*_k)^2
$$
Passing to the polar coordinates
\begin{equation}
    x(t) = \rho(t)\cos\theta(t), \qquad y(t) = -\rho(t)\sin\theta(t), \quad \text{for all }\, t \in I, \label{eq:3.22}
\end{equation}
we have that $(\rho(t), \theta(t))$ satisfies the differential system
\begin{equation}
\label{eqs}
    \begin{cases}
        \begin{aligned}
            &\rho' = X(t,-\rho\sin\theta)\cos\theta + Y(t,\rho\cos\theta)\sin\theta=: R(t,\rho, \theta) \\
            &\theta' = \frac{Y(t,\rho\cos\theta)\cos\theta - X(t,-\rho\sin\theta)\sin\theta}{\rho} := \Theta(t,\rho,\theta).
        \end{aligned}
    \end{cases}
\end{equation}
Moreover, $\rho(0) = \rho_k^*$ and, without loss of generality, 
$$\theta(0) = \theta_0 \in [0,2\pi).$$
With this, we only have to prove that
$$\theta(\tau) - \theta(0) > k \pi.$$

To do so, and since our hypotheses are only local, namely for
$x,y\in(-\delta,\delta)$, we first produce an estimate on the angular
displacement as long as the solution remains in a suitable
neighborhood of the origin. To this aim, we define
$\sigma\in(0,\tau]$ as the maximal value such that
\[
\check r_k\leq \rho(t)\leq \widehat r_k,
\qquad \text{for all }t\in[0,\sigma].
\]
By the choice of $\widehat r_k$, we have
\[
|x(t)|,|y(t)|\leq \rho(t)\leq \widehat r_k<\delta,
\qquad t\in[0,\sigma],
\]
and therefore estimates \eqref{first} and \eqref{second} hold throughout
$[0,\sigma]$.

\medskip
We now argue differently according to the fact that $\sigma = \tau$ or $\sigma < \tau$.

\begin{itemize}
    \item Suppose that $\sigma = \tau$. In this case, from \eqref{eqs}, after integrating we obtain
    $$\begin{aligned}\theta(\tau)-\theta(0) = \int_0^\tau \frac{Y(t,x)x(t) + X(t,y)y(t)}{x^2(t) + y^2(t)} dt &\geq \delta_k^* \tau \\
    &> 2\pi k 
    > k \pi. \end{aligned}$$
    Hence, the thesis follows.
    \item Assume that $\sigma < \tau$. In this case, the maximality of $\sigma$ implies that
    $$
    \check{r}_k \leq \rho(t) \leq \hat{r}_k, \quad \text{for all }\, t \in [0,\sigma] \quad \text{and} \quad \rho(\sigma) \in \{\check{r}_k, \hat{r}_k\}.
$$
We consider the case in which $\rho(\sigma) = \hat{r}_k$. The treatment of the other situation is completely symmetric (involving the consideration of $r_-$ instead of $r_+$) and thus is omitted.

We want to prove that
\begin{equation}
    \theta(\sigma) - \theta(0) > (k+1)\pi \label{eq:3.25}
\end{equation}
holds. Indeed, from \eqref{eq:3.25} one can conclude easily by observing that $\theta(\tau) - \theta(\sigma) > -\pi$ (as a consequence of {\cite[Lemma 3.1]{BoZa13}}, thanks to \eqref{third}) and therefore $\theta(\tau) - \theta(0) > k\pi$, yielding the desired result. If, by contradiction, we suppose that \eqref{eq:3.25} is not true, from $\theta'(t) > 0$ for all $t \in [0,\sigma]$, we obtain
$$
    0 < \theta(t) - \theta(0) < (k+1)\pi, \quad \text{for all }\, t \in\, ]0,\sigma[\,.
$$
At this point we note that the first equation in \eqref{eqs} can be written as
$$
    \rho' =  \rho S(t, \rho\cos\theta, -\rho\sin\theta),
$$
where
$$
    S(t,x,y) := \frac{X(t,y)x -Y(t,x)y}{x^2+y^2},
$$
while the second equation in \eqref{eqs} is of the form
$$
    \theta' = U(t, \rho\cos\theta, -\rho\sin\theta),
$$
with
$$
    U(t,x,y) := \frac{Y(t,x)x + X(t,y)y}{x^2+y^2}.
$$
Therefore, from the sign conditions \eqref{first} and \eqref{second}, we find that
\begin{equation}
    \frac{S(t,x,y)}{U(t,x,y)} \leq \mathcal{M}_+(x,y), \quad 
    \text{for all }\, t \in [0,\tau],\ x,y \in (-\delta,\delta),\ x^2+y^2 > 0. \label{eq:3.26}
\end{equation}
Let $\gamma : \mathbb{R} \to \mathbb{R}_0^+$ be the solution of \eqref{eq:3.18} with $r(\theta_0) = \rho_k^*$, that is 
$\gamma(\theta) := r_+(\theta; \theta_0, \rho_k^*)$. From the definition of $M_k^*(\rho_0)$ and the choice of $\hat{r}_k$ in \eqref{rj}, we can find $\varepsilon > 0$ such that
$$
    \gamma(\theta) < \hat{r}_k, \quad \text{for all }\, \theta \in [\theta_0 - \varepsilon,\, \theta_0 + 2(k+1)\pi + \varepsilon]\,.
$$
By the definition of $\gamma$ and \eqref{eq:3.26} it follows that
$$
    R(t, \gamma(\theta), \theta) \leq \gamma'(\theta)\Theta(t, \gamma(\theta), \theta), \quad
    \text{for all }\, t \in [0,\tau],\ \theta \in [\theta_0 - \varepsilon,\, \theta_0 + 2(k+1)\pi + \varepsilon]\,.
$$
Using {\cite[Corollary 5.1]{BoZa13}} with the positions $\theta_1 := \theta_0 - \varepsilon$, $\theta_2 := \theta_0 + 2(k+1)\pi + \varepsilon$ and $I := [0,\tau]$, we easily conclude that
$$
    \rho(t) \leq \gamma(\theta(t)), \quad \text{for all }\, t \in [0,\sigma]
$$
and hence
$$
    \rho(t) \leq M_{k+1}^* < \hat{r}_k \quad \text{for all }\, t \in [0,\sigma],
$$
thus contradicting the hypothesis that $\rho(\sigma) = \hat{r}_k$.

With a similar argument, one can check that \eqref{eq:3.25} holds if $\rho(\sigma) = \check{r}_k$. Hence, in any case, the conclusion follows.
\end{itemize}

\end{proof}

Additionally, we prove a variant of Proposition \ref{Prop21}, in which we consider a reversed time direction:
\begin{proposition}
\label{timeinverse}
Let $\delta>0$ and let $a_i, b_i : [-\delta, \delta] \to \mathbb{R}$, with $i = 1, 2$, be Lipschitz continuous functions such that
\begin{equation}
    0 < a_1(s)s \leq b_1(s)s \quad \text{and} \quad 0 < b_2(s)s \leq a_2(s)s, 
    \quad \text{for every } s \neq 0. \label{eq:2.31}
\end{equation}
Then, for every positive integer $k$, there exists $\hat \rho_k>0$ such that, for every $\rho^*_k \in (0,\hat\rho_k]$, there exist $\tau_k^* > 0$ such that, for every interval $I = [t_0, t_1] \subset \mathbb{R}$ with $t_1 - t_0 > \tau_k^*$ and for every locally Lipschitz continuous functions $X, Y : I \times \mathbb{R} \to \mathbb{R}$ satisfying
\begin{equation}
\label{first1}
    a_1(y)y \leq X(t,y)y \leq b_1(y)y, \quad \text{for every } t \in I,\ y \in (-\delta, \delta)
\end{equation}
and
\begin{equation}
\label{second2}
    b_2(x)x \leq Y(t,x)x \leq a_2(x)x, \quad \text{for every } t \in I,\ x \in (-\delta, \delta),
\end{equation}
\begin{equation}\label{third2}
\hspace{-0.5em} 0 <Y(t,x) x ,\quad \quad \quad \quad \quad \,\text{for every } t\in I, \ x \not= 0    
\end{equation}
it holds that any solution $(x(t), y(t))$ of \eqref{hamiltonian}, defined on $I$ and satisfying $x(t_1)^2 + y(t_1)^2 = (\rho_k^*)^2$, fulfills $x(t)^2 + y(t)^2 > 0$ for every $t \in I$ and
$$
    \theta(t_1) - \theta(t_0) > k\pi.
$$
\end{proposition}

\begin{proof}
Let $(x(t),y(t))$ be a solution of \eqref{hamiltonian} defined on $I=[t_0,t_1]$. Define
$$
\tilde x(t) = y(-t), \quad \tilde y(t) = x(-t),
$$
for $t \in \tilde I := [-t_1,-t_0]$. A direct computation shows that $(\tilde x,\tilde y)$ satisfies
\begin{equation} \label{hamiltoniantilde} 
\begin{cases}
\tilde x'(t) = Y(-t, \tilde y(t)) =: \tilde X(t, \tilde y(t)) \\
\tilde y'(t) =  - X(-t,\tilde x(t)) =: -\tilde Y(t, \tilde x(t)) ,
\end{cases}
\end{equation}
that is, a system of the same form as \eqref{hamiltonian} on $\tilde I$. Moreover, up to the relabeling
$$
\tilde a_1(s) = b_2(s), \quad \tilde a_2(s)= b_1(s), \quad 
\tilde b_1(s) = a_2(s), \quad \tilde b_2(s) = a_1(s),
$$
one verifies that $\tilde X$ satisfies \eqref{first} and \eqref{third}, and that $\tilde Y$ satisfies \eqref{second}. Hence, one can apply Proposition \ref{Prop21} to \eqref{hamiltoniantilde}.

We now introduce polar coordinates for $(\tilde x,\tilde y)$:
\begin{equation}
\label{polartilde}
\tilde x(t) = \tilde \rho(t) \cos \tilde \theta(t), \quad \tilde y(t) = - \tilde \rho(t) \sin \tilde \theta (t).
\end{equation}
If $|\tilde I| = |I| > \tau_k^*$ and
$$
\tilde x^2(-t_1) + \tilde y^2(-t_1) = (\rho_k^*)^2 = x^2(t_1) + y^2(t_1),
$$
then Proposition \eqref{Prop21} yields
\begin{equation}
\label{des}
\tilde \theta(-t_0) - \tilde \theta (-t_1) > k \pi.
\end{equation}

Finally, we explicitly state the relation between the angular variables. From the definition of $(\tilde x,\tilde y)$, \eqref{polartilde} and \eqref{eq:2.2}, we obtain
$$
\tilde x(t) = y(-t) = -\rho(-t)\sin\theta(-t), \quad 
\tilde y(t) = x(-t) = \rho(-t)\cos\theta(-t).
$$
Comparing with \eqref{polartilde}, it follows that $\tilde \rho(t)=\rho(-t)$ and
$$
\cos \tilde \theta(t) = -\sin \theta(-t), \quad 
-\sin \tilde \theta(t) = \cos \theta(-t),
$$
which implies
$$
\tilde \theta (t) = - \frac{\pi}{2} - \theta (-t).
$$
Substituting this identity into \eqref{des}, we obtain
$$
\theta(t_1) - \theta(t_0) > k \pi,
$$
which concludes the proof.
\end{proof}

\end{document}